\documentclass{amsart}
\usepackage{amsmath,amssymb,amsthm,amsfonts,enumerate,array,color,lscape,fancyhdr,layout,pst-all}
\usepackage[english]{babel}
\usepackage[latin1]{inputenc}
\usepackage{young}

\usepackage[hcentermath]{youngtab}
\usepackage{graphicx}
\usepackage{float}
\usepackage{pstricks}

\newcounter{numberofremark}
\newcommand\nothing[1]{}
\usepackage[active]{srcltx}
\usepackage[hyperindex,bookmarksnumbered,plainpages]{hyperref}

\newcommand{\dcl}{\DeclareMathOperator}
\dcl\cdet{cdet} \dcl\Sp{Specm} \dcl\depth{depth} \dcl\im{Im} \dcl\Span{span} \dcl\Ker{Ker} \dcl\Specm{Specm}
\dcl\Supp{Supp} \dcl\codim{codim} \dcl\Y{Y} \dcl\gl{\mathfrak{gl}}    \dcl\U{U} \dcl\T{T}
\dcl\qdet{qdet} \dcl\sgn{sgn} \dcl\gr{gr} \dcl\diag{diag}
\dcl\g{\mathfrak{g}} \dcl\C{\mathbb C} \dcl\dd{{\mathrm d}}
   
   \dcl\ch{ch}

\newcommand\sn{{\mathsf n}}
\newcommand\sm{{\mathsf m}}

\newcommand\Ga{{\Gamma}}

\newlength\yStones
\newlength\xStones
\newlength\xxStones

\makeatletter
\def\Stones{\pst@object{Stones}}
\def\Stones@i#1{%
  \pst@killglue%
  \begingroup%
  \use@par%
  \setlength\xxStones{\xStones}%
  \expandafter\Stones@ii#1,,\@nil
  \endgroup
  \global\addtolength\xStones{0.6cm}%
  \global\addtolength\yStones{-7.5mm}}%
\def\Stones@ii#1,#2,#3\@nil{%
  \rput(\xxStones,\yStones){%
    \psframebox[framesep=0]{%
      \parbox[c][6mm][c]{11mm}{\makebox[11mm]{$#1$}}}}%
  \addtolength\xxStones{1.2cm}%
  \ifx\relax#2\relax\else\Stones@ii#2,#3\@nil\fi}
\makeatother

\def\Stone#1{\fbox{\makebox[13mm]{\strut#1}}\kern2pt}

\newtheorem{theorem}{Theorem}[section]
\newtheorem{lemma}[theorem]{Lemma}
\newtheorem{corollary}[theorem]{Corollary}
\newtheorem{proposition}[theorem]{Proposition}

\newtheorem{remark}[theorem]{Remark}

\newtheorem{definition}[theorem]{Definition}

\begin{document}
\title{New singular  Gelfand-Tsetlin $\displaystyle \mathfrak{gl} (n)$-modules of index $2$}

\author{Vyacheslav Futorny}

\address{Instituto de Matem\'atica e Estat\'istica, Universidade de S\~ao
Paulo,  S\~ao Paulo SP, Brasil} \email{futorny@ime.usp.br,}
\author{Dimitar Grantcharov}
\address{\noindent
University of Texas at Arlington,  Arlington, TX 76019, USA} \email{grandim@uta.edu}
\author{Luis Enrique Ramirez}
\address{Universidade Federal do ABC,  Santo Andr\'e-SP, Brasil} \email{luis.enrique@ufabc.edu.br,}

\begin{abstract}
Singular Gelfand-Tsetlin modules of index 2 are modules whose tableaux bases may have singular pairs but no singular triples of entries on each row. In this paper we construct singular Gelfand-Tsetlin modules for arbitrary singular character of index $2$.  Explicit bases of derivative tableaux and the action of the generators of 
$\mathfrak{gl}(n)$ are given for these modules.  Our construction leads to new families of irreducible Gelfand-Tsetlin modules and also provides tableaux bases for some simple Verma modules. 
\end{abstract}

\maketitle

\section{Introduction} \label{sec-intro}
Gelfand-Tsetlin bases are among the most remarkable discoveries of the representation theory of classical Lie algebras. Originally introduced in \cite{GT},  these bases provide a convenient tableaux realization of every simple finite-dimensional representation of the Lie algebra $\mathfrak{gl} (n)$, as well as explicit formulas for the action of the generators of $\mathfrak{gl} (n)$.  The explicit nature of the Gelfand-Tsetlin formulas  inevitably raises the question of what infinite-dimensional modules admit  tableau bases. This question naturally initiated the theory of Gelfand-Tsetlin modules, a theory that has attracted considerable attention in the last 30 years and have been studied in  \cite{DFO2}, \cite{DFO3}, \cite{Maz1}, \cite{Maz2}, \cite{m:gtsb}, \cite{Zh}, among others. Gelfand-Tsetlin bases and modules are also related to  Gelfand-Tsetlin integrable systems that were first introduced for the unitary Lie algebra ${\mathfrak u}(n)$ by Guillemin and Sternberg in \cite{GS},  and later for the general linear Lie
algebra  $\mathfrak{gl}(n)$ by Kostant and Wallach in \cite{KW-1} and \cite{KW-2}.

We now define the main object of study in this paper. Consider a chain of embeddings 
$$\gl(1)\subset \mathfrak{gl}(2)\subset \ldots \subset \mathfrak{gl}(n).$$
The choice of embeddings is not essential but for simplicity we chose embeddings of principal submatrices.
Let $U=U(\mathfrak{gl}(n) )$ be the universal enveloping algebra of $\mathfrak{gl}(n)$, and let $\Ga$ be the {\it Gelfand-Tsetlin
subalgebra} of $U$, i.e. the subalgebra generated by the centers  of universal enveloping algebras of  all $\gl(i)$. Then ${\Ga}$ is a maximal commutative subalgebra of $U$ as well as a polynomial algebra in the $\displaystyle \frac{n(n+1)}{2}$ variables $c_{ij}$, where $c_{ij}$ is a degree $j$ element in the center of $U(\mathfrak{gl}(i))$, \cite{Zh}.  

A \emph{Gelfand-Tsetlin module} $V$ of $\mathfrak{gl}(n)$ is a Harish-Chandra $(U, \Ga)$-module, that is 
 $$V=\bigoplus_{\sm\in\Sp\Ga}V_{\sm},$$
where $$V_{\sm}=\{v\in V\; | \; \sm^{k}v=0 \text{ for some }k\geq 0\}.$$ The category of Harish-Chandra modules is a  subcategory of the category of all weight $\mathfrak{gl}(n)$-modules, i.e. modules that decompose as direct sum of modules over the standard Cartan subalgebra of $\mathfrak{gl}(n)$. Recall that the classification of all simple weight $\mathfrak{gl} (n)$-modules with finite-dimensional weight spaces is already completed, \cite{M}. Since the classification of arbitrary simple weight modules is out of reach for $n \geq 3$, the classification of simple Gelfand-Tsetlin $\mathfrak{gl} (n)$-modules seems to be the next fundamental classification problem one can try to solve. This, in addition to the connections with integrable systems and Yangians, gives us  another  motivation to study singular Gelfand-Tsetlin modules.  One should note also that the Gelfand-Tsetlin subalgebras are related to general hypergeometric functions on the complex Lie group $GL(n)$, \cite{Gr1},\cite{Gr2}, and to solutions of the Euler equation, \cite{Vi}.

Throughout the paper $n\geq 2$ and $T_n(\mathbb C)$ will stand for the space of  the following  \emph{Gelfand-Tsetlin
tableaux} with complex entries: \\
 
\begin{center}
\Stone{\mbox{ $v_{n1}$}}\Stone{\mbox{ $v_{n2}$}}\hspace{1cm} $\cdots$ \hspace{1cm} \Stone{\mbox{ $v_{n,n-1}$}}\Stone{\mbox{ $v_{nn}$}}\\[0.2pt]
\Stone{\mbox{ $v_{n-1,1}$}}\hspace{1.7cm} $\cdots$ \hspace{1.8cm} \Stone{\mbox{ $v_{n-1,n-1}$}}\\[0.3cm]
\hspace{0.2cm}$\cdots$ \hspace{0.8cm} $\cdots$ \hspace{0.8cm} $\cdots$\\[0.3cm]
\Stone{\mbox{ $v_{21}$}}\Stone{\mbox{ $v_{22}$}}\\[0.2pt]
\Stone{\mbox{ $v_{11}$}}\\
\medskip
\end{center}
 
 \
 
 We will identify  $ T_{n}(\mathbb{C})$ with the set  ${\mathbb C}^{\frac{n(n+1)}{2}}$ in the following way:
 to  $$
v=(v_{n1},...,v_{nn}|v_{n-1,1},...,v_{n-1,n-1}| \cdots|v_{21}, v_{22}|v_{11})\in {\mathbb C}^{\frac{n(n+1)}{2}}
$$ 
 we associate a tableau   $T(v)\in T_{n}(\mathbb{C})$ as above. It is important to distinguish $v$  and $T(v)$ since they are vectors in non-isomorphic vector spaces as explained below.

For a fixed element 
$v=(v_{ij})_{j\leq i=1}^n$  in $T_n(\mathbb C)$ consider a set 
$$v + T_{n-1}(\mathbb Z)=\{v+M\; | \; M=(m_{ij})_{j\leq i=1}^n\in T_n(\mathbb Z), m_{nk}=0 \,, k=1, \ldots, n \}.$$
Henceforth we define $V(T(v))$ to be the  complex vector space with basis the set $v + T_{n-1}(\mathbb Z)$, i.e. $V(T(v)) = \bigoplus_{w \in v + T_{n-1}(\mathbb Z)} {\mathbb C} T(w)$. Note that $T(v+w) \neq T(v) + T(w)$ in $V(T(v))$ (even if $w,v+w \in v + T_{n-1}(\mathbb Z)$).

 To every  $w\in v + T_{n-1}(\mathbb Z)$ we 
 associate the maximal ideal $\sm_{w}$ of $\Ga$ generated by $c_{ij} - \gamma_{ij}(w)$, where $\gamma_{ij}(w)$ are symmetric polynomials defined in (\ref{def-gamma}). Note that the correspondence $w\mapsto \sm_{w}$ is not one-to-one, but for  a given maximal ideal $\sm$  there are finitely many  $w\in v + T_{n-1}(\mathbb Z)$ with $\sm_{w} = \sm$ (see Remark \ref{correspondence between characters and tableaux} for details).    We will call the set of all such $w$, the {\it fiber of $\sm$ in} $v + T_{n-1}(\mathbb Z)$ and denote it by $\widehat{\sm}$. 
  
From now on we set $G:=S_{n}\times\cdots\times S_{1}$. Note that $G$ and $T_{n-1} (\mathbb Z)$ act naturally on $T_{n} (\mathbb C)$ (see Section \ref{sec-prelim} for the explicit action formulas). 
 For each  $w \in T_{n} (\mathbb C)$,  the fiber $\widehat{\sm}_{w}$ of $\sm_{w}$  coincides with  the intersection of $V(T(v))$ and the orbit $Gw$ of the group action of $G$ on $T_{n} (\mathbb C)$.

 In this paper we address the following problem:

\medskip
\noindent {\bf Problem:}  {\it  Given $v\in {\mathbb C}^{\frac{n(n+1)}{2}}$ is it possible to define a non-trivial Gelfand-Tsetlin $\mathfrak{gl}(n)$-module structure on $V(T(v))$, so that  $V(T(v))_{\sm} = \bigoplus_{w\in (v + T_{n-1}(\mathbb Z)) \cap \widehat{\sm}} {\mathbb C} T(w)$?}

\medskip
 Note that by fixing $v$ we prescribe a basis of $V(T(v))_{\sm}$ and the action of $\mathfrak{gl}(n)$ on $V(T(v))$ should
match that prescription. Also note that if $v-v' \in T_{n-1}(\mathbb Z)$ the vectors spaces $V(T(v))$ and $V(T(v'))$ are isomorphic, but the $\mathfrak{gl} (n)$-modules $V(T(v))$ and $V(T(v'))$ are not necessarily isomorphic, see Theorem B(ii) below.

\medskip
 The problem above was raised and studied by Gelfand and Graev in \cite{GG} and by Lemire and Patera (for $n=3$) in \cite{LP1}, \cite{LP2}.  Apparently the main challenge when solving  this problem occurs when two entries  in  one row of $T(v)$ have integer difference. We will say that $v$ is \emph{singular of index} $m\geq 2$ if:
 \begin{itemize}
 \item[(i)] there exists a row $k$, $1<k<n$, and $m$ entries  $v_{ki_1}, \ldots, v_{ki_m}$ on this row such that $v_{ki_j}-v_{ki_s}\in \mathbb Z$ for all $j,s\in \{1, \ldots, m\}$;
 \item[(ii)] $m$ is maximal with the property (i).
 \end{itemize} 
A  pair of entries $(v_{ki_j},v_{ki_s})$  such that $k>1$ and $v_{ki_j}-v_{ki_s}\in \mathbb Z$ is called a {\it singular pair}. We say that $v$ (and $T(v)$) is \emph{generic} if $v$  has no singular pairs.  For a generic  $v$, the  $\mathfrak{gl}(n)$-module structure on $V(T(v))$ was introduced in \cite{DFO3}. 
 
In \cite{FGR2} we initiated the  study of singular (i.e.  non generic) modules $V(T(v))$. This study consists of three 
steps. The first step is to look at singular tableaux $T(v)$ that contain a unique singular pair. This  case, called the {\it $1$-singular} case,  was treated in \cite{FGR2} and  is a particular case of a singularity of index $2$. By understanding just the $1$-singular case, we are able to complete the classification of all irreducible Gelfand-Tsetlin $\mathfrak{gl}(3)$-modules, \cite{FGR3}. In the present paper we make the next step in the study and address the case of  arbitrary singularity of index $2$. That is, any number of singular pairs (but not singular triples) and  multiple singular pairs in the same row are allowed. The transition of a unique singular pair to a general singularity of index $2$ turned out to be not straightforward  and  requires a more sophisticated  theory of differential operators and divided differences of tableaux.  The methods developed in this paper will be crucial for completing the last step in our study - defining singular Gelfand-Tsetlin modules $V(T(v))$ of arbitrary index. The last case will be addressed in a subsequent paper. One should note that, as a straightforward consequence of the results in the present paper, we obtain new tableaux bases of a large family of irreducible Verma modules of $\mathfrak{gl} (n)$.

For the rest of the introduction we fix $v$ to be singular of index $2$. 
If $w\in v + T_{n-1}(\mathbb Z)$ is such that $w_{ki}=w_{kj}$ for some $k,i,j$, $1<k<n$, $i \neq j$, we say that $w$ lies on the {\it critical hyperplane} $x_{ki}-x_{kj}=0$ of $T_{n} (\mathbb C)$. We also  say that $w$ is {\it maximal critical} if it lies on the intersection of all possible critical hyperplanes corresponding to elements in $v + T_{n-1}(\mathbb Z)$. 

 Now we state our first  main result.\\

\noindent {\bf Theorem A.} {\it Let $v \in {\mathbb C}^{\frac{n(n+1)}{2}}$ be singular of index $2$ and let $t$ be the number of singular pairs of $v$. Then  
  $V(T(v))$ has a structure of a Gelfand-Tsetlin  $\mathfrak{gl}(n)$-module.     In particular, 
$$\dim  V(T(v))_{\sm}= 2^{t-k},$$ whenever $\widehat{\sm}$ lies on the intersection of $k$ critical hyperplanes, $0 \leq k \leq  t$, and 
 $\dim  V(T(v))_{\sm}= 1$ if $\widehat{\sm}$ consists of maximal critical points.}

\bigskip

Theorem A is proven by constructing a particular $\gl (n)$-action on $V(T(v))$ such that $V(T(v))_{\sm} = \bigoplus_{w \in (v + T_{n-1}(\mathbb Z))\cap \widehat{\sm}} {\mathbb C} T(w)$. The question whether the $\gl (n)$-action on $V(T(v))$ with this property is unique remains open.

In the subsequent Theorems B and C, and Conjectures 1--3, it is supposed that $V(T(v))$ is the $\gl(n)$-module constructed in the proof of Theorem A.

\noindent {\bf Theorem B.} {\it Let $v \in {\mathbb C}^{\frac{n(n+1)}{2}}$  be singular of index $2$.
 \begin{itemize}
 \item[(i)] If $v'\in v + T_{n-1}(\mathbb Z)$ has $s$ singular (but non-critical) pairs in row $k$, then $c_{k2}$ has an eigenvalue of geometric multiplicity $s+1$  on the  subspace $V(T(v))_{\sm_{v'}}$, and this is the largest geometric multiplicity of all eigenvalues of all elements $c_{kj}$, $1\leq j \leq k$, on the subspace $V(T(v))_{\sm_{v'}}$.
\item[(ii)]  $V(T(v))\simeq V(T(v'))$ if and only if there exists $\sigma\in G$ such that $v-\sigma(v')\in T_{n-1}(\mathbb Z)$, or equivalently, if $v$ and $v'$ are in the same orbit under the action of $G \ltimes T_{n-1} ({\mathbb Z})$ on $T_{n} (\mathbb C)$.
\item[(iii)] 
 Assume   that all singular pairs of $v$ belong to different rows, and   $v_{ij}-v_{i-1,k}$ is not integer for all possible indexes $i,j,k$. Then $V(T(v))$ is irreducible. 
\end{itemize} }
\bigskip

\noindent {\bf Conjecture 1.} The  condition  that $v_{ij}-v_{i-1,k}$ is not integer is both necessary and sufficient for the 
irreducibility of $V(T(v))$ in Theorem B.\\

The above conjecture is known to be true for $n=2$, $n=3$, and for  $1$-singular tableaux $T(v)$, \cite{GoR}.
\bigskip

From Theorem B we obtain many explicit examples of new irreducible singular Gelfand-Tsetlin modules together with information about their structure. In particular, we can compute an important invariant for  these irreducible modules: their {\it Gelfand-Tsetlin degree}, namely the maximum Gelfand-Tsetlin multiplicity that may appear. Furthermore, the generators $c_{ij}$ of $\Ga$ have a simultaneous canonical form on the subspaces $V(T(v))_{\sm}$ with largest Jordan cells of size $s+1$ where $s$ is the maximal number of singular pairs in one row.  All known examples so far concerned Jordan cells of size at most $2$ only.

Our last result addresses the Gelfand-Tsetlin theory properties of the modules $V(T(v))$.  It was shown in \cite{Ovs} that for every maximal ideal $\sm$ of $\Ga$, there is an irreducible Gelfand-Tsetlin module $M$ such that $M_{\sm} \neq 0$. Moreover, there exist only finitely many isomorphism classes of such modules.   If $\sm$ is generic then there is exactly one such isomorphism class   and its  Gelfand-Tsetlin degree is $1$.  On the other hand, if $V$ is an irreducible Gelfand-Tsetlin module then $\dim V_{\sm}$ is finite for all $\sm$ and is bounded by $1!2!\ldots (n-1)!$, \cite{FO2}.  The most interesting case certainly is  when $\sm$ is singular. In \cite{FGR2} we constructed irreducible   Gelfand-Tsetlin modules of Gelfand-Tsetlin degree 2, which is the highest possible degree in the case $n=3$. With the aid of Theorem B we obtain examples of irreducible modules of arbitrarily large degree.  More precise upper bound for the Gelfand-Tsetlin degree of the subquotients of $V(T(v))$ is listed in the next theorem, our third main result.

\bigskip

 \noindent {\bf Theorem C.} {\it Let $v \in T_{n} (\mathbb C)$ be singular of index $2$ and let $t$ be the number of  singular pairs of $L$.  The following hold for any $v'\in v + T_{n-1}(\mathbb Z)$.
  \begin{itemize}
 \item[(i)] There exists an irreducible subquotient $V$ of $V(T(v))$ such that $V_{\sm_{v'}}\neq 0$.
  The number of all such irreducible subquotients  of $V$ is bounded by $2^t$.
 \item[(ii)] $\dim \, V_{\sm_{v'}}\leq 2^{t-k}$ if $\sm_{v'}$ belongs to $k$ critical hyperplanes. In particular,  $\dim \, V_{\sm_{v'}}=1$ if $\hat{\sm}_{v'}$ consists of maximal critical points. 
 \item[(iii)] If each row contains at most one singular pair, then  the geometric multiplicities of all eigenvalues of any $c_{ij}$ are at most $2$.
 \item[(iv)] If $v$ has $s$ singular pairs in the $i$-th row then  the geometric multiplicities of all eigenvalues of any $c_{ij}$, $j=1, \ldots, i$, are at most $s+1$.
\end{itemize}}

\bigskip

\noindent {\bf Conjecture 2.} Any irreducible Gelfand-Tsetlin module $N$ with $N_{\sm_{v}}\neq 0$ for any singular  $v$ of index $2$ appears as a subquotient of $V(T(v))$.
 
 \medskip

Conjecture 2 is known to be true for $n=2$ and $n=3$.
For the case when there exists a unique singular pair in $v$ (i.e. $v$ is $1$-singular) this conjecture was stated in \cite{FGR2},  and recently proven in \cite{FGR4}.  

If $V$ is a Gelfand-Tsetlin module then define the Gelfand-Tsetlin character of $V$ as $$\ch_{GT} V=\sum_{\sm} (\dim V_{\sm}) e^{\sm}$$

It is an interesting question whether $\ch_{GT} V$ determines $V$.  The affirmative answer is known for generic and $1$-singular modules. We conjecture this in general:

 \medskip
 
\noindent {\bf Conjecture 3.} For any singular  $v$  of index $2$  and any two  irreducible subquotients $V$ and $W$ of $V(T(v))$, $V\simeq W$ if and only if 
 $\ch_{GT} V=\ch_{GT} W$.

 \medskip

The organization of the paper is as follows. In Section \ref{sec-prelim} we introduce the notation used in the paper and collect important results for finite-dimensional and generic Gelfand-Tsetlin modules. In Section \ref{section: derivative tableaux} we define the derivative tableaux corresponding to a singular vector $v$ of index $2$ and with their aid, define the space $V(T(v))$. The theorem that $V(T(v))$ is a $\mathfrak{gl} (n)$-module is stated and proven in Section \ref{sec-mod-str}. The formulas for the action of the generators of the Gelfand-Tsetlin subalgebra $\Gamma$ on  $V(T(v))$ are included in Section \ref{section: action of Gamma}.  The proofs of the three main results are given in Section \ref{sec-proofs}. In Section \ref{sec: Irreducible modules of index 2} we include examples of new irreducible Gelfand-Tsetlin $\mathfrak{gl} (n)$-modules, and, in particular, provide derivative tableaux realization of some irreducible Verma $\mathfrak{gl} (n)$-modules.
\bigskip

\noindent{\bf Acknowledgements.}  V.F. is
supported in part by  CNPq grant (301320/2013-6) and by 
Fapesp grant (2014/09310-5). D.G is supported in part by Simons Collaboration Grant 358245.  L.E.R. gratefully acknowledges the
hospitality and excellent working conditions at the University of Texas at Arlington, where part of this work was completed. L.E.R. is supported in part by Mathamsud grant 022/14. 
The authors are grateful to the referees for the numerous useful suggestions.

\section{Index of notations}
Below we list some notations that are frequently used in the paper under the section number they are  introduced first.\\
\S \ref{sec-intro} $T_n(\mathbb C)$, $T_{n-1}(\mathbb Z)$, $T(v)$, $V(T(v))$.\\
\S \ref{subsection:Gelfand-Tsetlin modules}. $U_{m}$, $Z_m$, $c_{mk}$, $\Gamma$, $M_{\sm}$.\\
\S \ref{subsection: Finite dimensional GT modules}. $\gamma_{mk}(v)$.\\
\S \ref{sec-gen-gt}. $T_{n}(\mathbb{C})_{\rm gen}$, $\widetilde{S}_{m}$, $\Phi_{kl}$, $e_{rs}(w)$, $\varepsilon_{rs}$.\\
\S \ref{section: derivative tableaux}. $\Sigma$, $R$, $\mathcal{S}^{0}$, $\mathcal{V}_{\rm gen}$, $R_{\Delta}$, $P_{\Delta}(v)$, $\Sigma_{I}$, $\tau_{r}$, $\tau_{\Delta}$, $\mathcal{H}$, $\overline{\mathcal{H}}$, $\mathcal{F}$, $\overline{\mathcal{F}}$, $\mathcal{D}_{I}$, $\mathcal{D}_{I}^{v}$, $v_{\Delta}$, ${\rm ev}(v)$, $\mathcal{D}_{I}T(v+z)$.\\
\S \ref{section: action of Gamma}. $\mathcal{L}_{a}^{(l)}$, $\mathcal{L}_{\geq a}^{(l)}$, $\mathcal{L}_{a}$, $\Sigma(m)$, $\Gamma_{m}$, $\prec_{\mathcal{D}}$, $\mathcal{D}_{J}T(v+z')\xrightarrow{g}\mathcal{D}_{I}T(v+z)$.\\
\S \ref{sec: Irreducible modules of index 2}. $\mbox{GT-}\deg (M)$, 
 {\rm gmult($\gamma |_{M_{\sm}}$)}.\\ 
\S \ref{subsection: Identities divided differences}. $f(v)^{\tau_{\Delta}}$, $f(v)_{\Delta}$.\\
\S \ref{subsection: Identities for e_rs}. $\Phi_{kl}(u,s)$, $\Phi_{kl}(i_r)$, $\Phi_{kl}(j_r)$, $\tau_{\Delta}^{\star}(\sigma)$, $\Omega_{0}$, $\Omega(\sigma)$, $\Omega_{\sigma}$, $\widetilde{\Omega}(\sigma)$, $\widehat{\Omega}(\sigma_1,\sigma_2)$, $\Phi_{(\sigma_1,\sigma_{2})}$.

\section{Preliminaries} \label{sec-prelim}

\subsection{Conventions}\label{subsection: conventions}
The ground field will be ${\mathbb C}$.  For $a \in {\mathbb Z}$, we write $\mathbb Z_{\geq a}$ for the set of all integers $m$ such that $m \geq a$.  
By $Re (z)$ we denote the real part of a complex number $z$, while $\lfloor x \rfloor$ stands for the greatest integer less than or equal to the real number $x$.

Let $\{E_{ij}\; | \; 1\leq i,j \leq n\}$  be the
standard basis of $\mathfrak{gl}(n)$ of elementary matrices. We fix the standard triangular decomposition and  the corresponding basis of simple roots of $\mathfrak{gl}(n)$.  The weights of $\mathfrak{gl}(n)$ will be written as $n$-tuples $(\lambda_1,...,\lambda_n).$ 

For a Lie algebra ${\mathfrak a}$ by $U(\mathfrak a)$ we denote the universal enveloping algebra of ${\mathfrak a}$. Throughout the paper $U = U(\mathfrak{gl}(n))$.  For a commutative ring $R$, by ${\rm Specm}\, R$ we denote the set of maximal ideals of $R$.

The transposition of the symmetric group $S_N$ interchanging $i$ and $j$ will be denoted by $(i,j)$. We set $G:=S_n\times S_{n-1} \times \cdots \times S_1$ and the $i$-th component of $\sigma \in G$ will be denoted by $\sigma[i]$.

\subsection{Gelfand-Tsetlin modules}\label{subsection:Gelfand-Tsetlin modules}
 Let  for $m\leqslant n$, $\mathfrak{gl}(m)$ be the Lie subalgebra
of $\mathfrak{gl}(n)$ spanned by $\{ E_{ij}\; | \;  i,j=1,\ldots,m \}$ and let   $U_{m}=U(\mathfrak{gl}(m))$. Let
$Z_{m}$ be the center of $U_{m}$. Then $Z_m$ is the polynomial
algebra in the $m$ variables $\{ c_{mk}\; | \; k=1,\ldots,m \}$,
\begin{equation}\label{equ_3}
c_{mk } \ = \ \displaystyle {\sum_{(i_1,\ldots,i_k)\in \{
1,\ldots,m \}^k}} E_{i_1 i_2}E_{i_2 i_3}\ldots E_{i_k i_1}.
\end{equation}

 The {\it Gelfand-Tsetlin
subalgebra} $\Ga$ is the subalgebra of $U$ generated by $\displaystyle \bigcup_{m=1}^n
Z_m$. Recall the definition of a Gelfand-Tsetlin module from the introduction. Namely, $M$ is a Gelfand-Tsetlin module if $M$ splits into  the direct sum of the  $\Ga$-modules $M_{\sm}=\{v\in M| \sm^{k}v=0 \text{ for some }k\geq 0\}$ indexed by the maximal ideals of $\Gamma$. The  {\it support} of a Gelfand-Tsetlin module $M$ is the set of all maximal ideals $\sm\in \Sp\Ga$ such that  $M_{\sm}\neq 0$. For any $\sm$ in the support of $M$,  the {\it Gelfand-Tsetlin multiplicity} of $\sm$ is  $\dim M_{\sm}$.

Note that any irreducible Gelfand-Tsetlin  module over $\mathfrak{gl}(n)$ is a weight module with respect to the
 standard Cartan subalgebra $\mathfrak h$ spanned by $E_{ii}$, $i=1,\ldots, n$. The converse is  not true in general (except for $n=2$), i.e. an irreducible weight module $M$ need not to be Gelfand-Tsetlin. However, it is the case when
the weight multiplicities of $M$ are finite.

We will use the following terminology for a Gelfand-Tsetlin module $M$, $\gamma \in \Gamma$, and $\sm \in \Specm \Gamma$:
\begin{itemize}
\item[(i)] The \emph{Gelfand-Tsetlin degree} (or, the \emph{GT-degree}) $\mbox{GT-}\deg (M)$ of $M$ is the supremum of all Gelfand-Tsetlin multiplicities of $M$, i.e. 
$$\mbox{GT-}\deg (M):= \sup \{ \dim M_{\sm} \; |\; \sm \in \Specm \Gamma\}.$$
\item[(ii)] The \emph{geometric multiplicity {\rm gmult($\gamma |_{M_{\sm}}$)} of  $\gamma$ in $M_{\sm}$}  is the size of the largest Jordan cell of the endomorphism $\gamma |_{M_{\sm}}$ on $M_{\sm}$.
\item[(iii)] The \emph{geometric GT-degree of $M$} is the maximum of  {\rm gmult($\gamma |_{M_{\sm}}$)}  over all $\gamma \in \Gamma$ and all $\sm \in \Specm \Gamma$.
\end{itemize}


The action of $G = S_n \times S_{n-1} \times \cdots \times S_1$ on $T_n (\mathbb C)$ is given by the formula:
 \begin{equation} \label{eq-g-action}
\sigma(v):=(v_{n,\sigma^{-1}[n](1)},\ldots,v_{n,\sigma^{-1}[n](n)}|\ldots|v_{1,\sigma^{-1}[1](1)}).
\end{equation}
where $v\in T_{n}(\mathbb{C})$ and $\sigma\in G$. In addition to the $G$-action, we have another important action on $T_n (\mathbb C)$: the action by translations of $T_{n-1} (\mathbb Z)$. The two actions can be combined into one action of the semidirect product $G \ltimes T_{n-1} ({\mathbb Z})$.

For $1\leq j \leq i \leq n-1$, $\delta^{ij} \in T_{n-1}(\mathbb{Z})$ is defined by  $(\delta^{ij})_{ij}=1$ and all other $(\delta^{ij})_{k\ell}$ are zero.

We have the following important types of tableaux, and equivalently, of vectors in $T_n (\mathbb C)$.

\begin{definition} Let $v \in T_n (\mathbb C)$.
\begin{itemize}
\item[(i)] We call $T(v)$ a \emph{standard tableau} if 
$$v_{ki}-v_{k-1,i}\in\mathbb{Z}_{\geq 0} \hspace{0.3cm} and \hspace{0.3cm} v_{k-1,i}-v_{k,i+1}\in\mathbb{Z}_{> 0}, \hspace{0.1cm}\text{ for all } 1\leq i\leq k\leq n.$$
\item[(iii)] We call $T(v)$ a \emph{regular tableau} (and $v$ a \emph{regular vector}) if $v_{rs}-v_{r-1,t}\notin\mathbb{Z}$ for any $r,\ s,\ t$. 
\end{itemize}
\end{definition}

\subsection{Finite-dimensional Gelfand-Tsetlin modules}\label{subsection: Finite dimensional GT modules}
The standard Gelfand-Tsetlin tableaux play a key role in the description of a Gelfand-Tsetlin basis of finite-dimensional representations of $\mathfrak{gl}(n)$. Below we recall this classical result. 

\begin{theorem}[\cite{GT}]\label{Gelfand-Tsetlin theorem}
Let $L(\lambda)$ be the finite-dimensional irreducible $\mathfrak{gl}(n)$-module of highest weight $\lambda=(\lambda_{1},\ldots,\lambda_{n})$. Then the set of all standard tableaux $T(v)$ with fixed top row $v_{ni}=\lambda_i-i+1$, $i=1,\ldots,n$ forms a basis of $L(\lambda)$. Moreover,  the action of the generators of $\mathfrak{gl}(n)$ on $L(\lambda)$ is given by the  \emph{Gelfand-Tsetlin formulas}:

$$E_{k,k+1}(T(v))=-\sum_{i=1}^{k}\left(\frac{\prod_{j=1}^{k+1}(v_{ki}-v_{k+1,j})}{\prod_{j\neq i}^{k}(v_{ki}-v_{kj})}\right)T(v+\delta^{ki}),$$

$$E_{k+1,k}(T(v))=\sum_{i=1}^{k}\left(\frac{\prod_{j=1}^{k-1}(v_{ki}-v_{k-1,j})}{\prod_{j\neq i}^{k}(v_{ki}-v_{kj})}\right)T(v-\delta^{ki}),$$

$$E_{kk}(T(v))=\left(k-1+\sum_{i=1}^{k}v_{ki}-\sum_{i=1}^{k-1}v_{k-1,i}\right)T(v),$$
where if the sum of $E_{k,k+1}(T(v))$ or $E_{k+1,k}(T(v))$ contains a summand with a non-standard $T(v\pm\delta^{ki})$, then the  summand is assumed to be zero. 

\end{theorem}

Another important result is that a module defined by the Gelfand-Tsetlin formulas  is a Gelfand-Tsetlin module. In particular, we have the following.

\begin{theorem}[\cite{Zh}]\label{action of Gamma on fd modules} The action of the generators of $\Gamma$ on a finite-dimensional module $L(\lambda)$ is given by the following formulas:
$$c_{mk}(T(v))   = \  \gamma_{mk} (v)T(v),$$ where
\begin{equation} \label{def-gamma}
\gamma_{mk} (v): = \ \sum_{i=1}^m
(v_{mi}+m-1)^k \prod_{j\ne i} \left( 1 -
\frac{1}{v_{mi}- v_{mj}} \right),
\end{equation} 
with the generating function
$$1-\sum_{k=0}^{\infty}\gamma_{mk}(v) u^{-k-1} =\prod_{i=1}^{m}\frac{u - v_{mi} -m}{u - v_{mi} - m + 1}.$$

\end{theorem}

\begin{remark}\label{correspondence between characters and tableaux}
 There is a natural correspondence between the set $\Gamma^*$ of characters $\chi : \Gamma \to \mathbb{C}$ 
 (and, hence, of maximal ideals of $\Gamma$, $\sm=\Ker \, \chi$)
 and the set of Gelfand-Tsetlin tableaux. In fact, to obtain a Gelfand-Tsetlin tableau from a character $\chi$ we find a solution $v = (v_{ij})$ of the system of equations
$$\{\gamma_{mk}(v) =\ \chi(c_{mk})\}_{1\leq k\leq m\leq n}$$
Conversely, for every tableau $T(v)$ we associate $\chi \in \Gamma^*$ by defining  $\chi(c_{mk})$ via the above equations. It is clear that each tableau defines such a character uniquely. On the other hand, a
tableau is defined by a character uniquely up to a permutation in $G$.
\end{remark}

\subsection{Generic Gelfand-Tsetlin modules}\label{sec-gen-gt} Since  the coefficients in the Gelfand-Tsetlin formulas in Theorem \ref{Gelfand-Tsetlin theorem} are rational functions on the entries of the tableaux, it is natural to extend the Gelfand-Tsetlin construction to more general modules.  When all denominators are nonintegers, one can use the same formulas and  define a new class of infinite-dimensional {\it generic}  Gelfand-Tsetlin 
$\mathfrak{gl}(n)$-modules (cf. \cite{DFO3}, Section 2.3). Recall the definition of $V(T(v))$ from the introduction. Then $V(T(v))$ is a generic Gelfand-Tsetlin module 
with action of the generators of $\mathfrak{gl}(n)$ given by the Gelfand-Tsetlin formulas.  All Gelfand-Tsetlin  multiplicities of $V(T(v))$ are $1$.

Denote by $T_{n}(\mathbb{C})_{\rm gen}$ the set of all generic vectors in $T_{n}(\mathbb{C})$. By $\widetilde{S}_m$ we denotes the subset of $S_m$ consisting of the transpositions $(1,i)$, $i=1,...,m$. For $k<\ell $, set $\Phi_{k\ell} =  \widetilde{S}_{\ell-1} \times\cdots\times \widetilde{S}_{k}$. For $k>\ell$ we set $\Phi_{k\ell} = \Phi_{\ell k}$. Finally we let $\Phi_{\ell \ell} = \{ \mbox{Id}\}$.  Every $\sigma$ in $\Phi_{k\ell }$ will be written as a $|k-\ell|$-tuple  of transpositions $\sigma[i]$ (recall that $\sigma[i]$ is the $i$-th component of $\sigma$). Also, identify every $\sigma\in \Phi_{k\ell} $ as an element of $G = S_n \times \cdots \times S_1$ by letting  $\sigma [i] = \mbox{Id}$ whenever  $i<\min\{k,\ell\}$ or  $ i> \max\{k,\ell\}-1$.

\begin{remark}
Gelfand-Tsetlin formulas are given for the generators of $\gl(n)$ as a Lie algebra. For convenience we will write explicitly the action of any $E_{rs}\in\gl(n)$ in terms of permutations. The corresponding coefficients for the action of $E_{rs}$ can be obtained by computing the action of $[E_{r,s-1},E_{s-1,s}]$ and induction on $|r-s|$.   
\end{remark}

\begin{definition}\label{definition of coefficients e_rs}
For each generic vector $w$ and any $1\leq r, s\leq n$ we define
$$
e_{rs}(w):=
\begin{cases}
-\frac{\prod_{j=1}^{s}(w_{s-1,1}-w_{s,j})}{\prod_{j=2}^{s-1}(w_{s-1,1}-w_{s-1,j})}\displaystyle\prod_{j=r}^{s-2}\left(\frac{\prod_{t=2}^{j+1}(w_{j1}-w_{j+1,t})}{\prod_{t=2}^{j}(w_{j1}-w_{jt})}\right), & \text { if }\ \ r<s,\\
\frac{\prod_{j=1}^{s-1}(w_{s1}-w_{s-1,j})}{\prod_{j=2}^{s}(w_{s1}-w_{sj})}\displaystyle\prod_{j=s+2}^{r}\left(\frac{\prod_{t=2}^{j-2}(w_{j-1,1}-w_{j-2,t})}{\prod_{t=2}^{j-1}(w_{j-1,1}-w_{j-1,t})}\right), &  \text { if }\ \ \ r>s,\\
r-1+\sum\limits_{i=1}^{r}w_{ri}-\sum\limits_{i=1}^{r-1}w_{r-1,i}, &  \text { if }\ \ \ r=s,
\end{cases}
$$
\end{definition}

Let $1 \leq r < s \leq n-1$. Set $
\varepsilon_{rs}:=\delta^{r,1}+\delta^{r+1,1}+\ldots+\delta^{s-1,1} \in  T_{n}(\mathbb{Z}),$ $\varepsilon_{rr}=0$ and  $\varepsilon_{sr}=- \varepsilon_{rs}$.

Note that for any $w\in T_{n}(\mathbb{C})$ and $\sigma\in\Phi_{k\ell }$, we define $\sigma(w)$ according to (\ref{eq-g-action}). We have the following important result for generic Gelfand-Tsetlin modules. 
\begin{proposition}[\cite{FGR2}]\label{coefficients e_ij} Let $v \in T_n(\mathbb C)$ be generic. Then the  $\mathfrak{gl}(n)$-module structure on $V(T(v))$ is defined by the formulas:
\begin{equation}\label{formula of coefficients e_ij}
E_{m\ell} (T(v+z))= \sum_{\sigma \in \Phi_{m\ell}} e_{m\ell} (\sigma (v+z)) T(v+z+\sigma(\varepsilon_{m\ell})),
\end{equation}
for $z \in T_{n-1}({\mathbb Z})$ and $1 \leq m ,\ \ell \leq n$. Moreover, $V(T(v))$ is a Gelfand-Tsetlin module with action of $\Ga$  given by the formulas in Theorem \ref{action of Gamma on fd modules}.
\end{proposition}

\section{Derivative tableaux}\label{section: derivative tableaux}
In this and next sections we define an appropriate module structure on the space $V(T(v))$. To do this we distinguish certain {\it derivative} tableaux in the spanning set of $V(T(v))$. The action of $\mathfrak{gl}(n)$
on derivative tableaux  will be different from the action on the other (ordinary) tableaux. One reason for that
 is the following. Suppose $v \in T_{n}(\mathbb{C})$ is such that $v_{ki}-v_{kj}\in \mathbb Z$ for some $1<k<n$ and $i\neq j$. 
 Then the tableaux $T(v)$ and $T(v+z)$ define the same maximal ideal $\sm$ of $\Gamma$ for some $z\in  T_{n-1}(\mathbb{Z})$, that is they are indistinguishable by $\Gamma$.  In addition, if $v_{ki}-v_{kj}\in \mathbb Z$ the action of $\gl(n)$ on some $T(v+z)$ described in (\ref{formula of coefficients e_ij}) will involve zero denominators.

\begin{definition}\label{Def of 1^t singular}
A vector $w\in T_{n}(\mathbb{C})$ is called \emph{$t$-singular of index $2$} if there are exactly $t$ singular pairs and no singular triples, that is, if there are $(k_{r},i_{r},j_{r})$, $r=1,\ldots,t$, such that:

\begin{itemize}
\item[(i)] $2\leq k_{1}\leq\ldots\leq k_{t}\leq n-1$.
\item[(ii)] $1\leq i_{r}<j_{r}\leq k_{r}$ for each $r= 1,\ldots,t$.
\item[(iii)] If $k_{r} = k_{s}$ for some $r,s$, then $\{ i_r, j_r\} \cap \{i_s, j_s \} = \emptyset$.
\item[(iv)] For any $r= 1,\ldots,t$ we have $w_{k_{r},i_{r}}-w_{k_{r},j_{r}}\in\mathbb{Z}$ and $w_{ki}-w_{kj}\notin \mathbb{Z}$ for any $(k,i,j)\notin\{(k_{r},i_{r},j_{r})\; | \; r=1,\ldots,t\}$.
\end{itemize}
\end{definition}

\medskip
A maximal ideal $\sn $ of $\Ga$ is called {\it $t$-singular} of index $2$ if $v=v_{\sn}$ is $t$-singular of index $2$ for one choice (hence for all choices) of $v$ in $\widehat{\sn}$.
  A Gelfand-Tsetlin module $M$ will be called {\it $t$-singular Gelfand-Tsetlin module of index $2$} if any $\sn$ in the Gelfand-Tsetlin support of $M$ is $t$-singular  of index $2$.

In the following we fix some notation for the rest of the paper. 
\begin{definition}
From now on, $t$ and $\{(k_{r},i_{r},j_{r})\; | \; r=1,\ldots,t\}$ will be fixed. We also fix a $t$-singular vector $v$ of index $2$ such that $v_{k_{r},i_{r}}=v_{k_{r},j_{r}}$ for every $r=1,...,t$. Furthermore, we will denote $\Sigma:=\{1,\ldots,t\}$ and $R:=(\{(i_{1},j_{1})\},\ldots,\{(i_{t},j_{t})\})$.
\end{definition}

Note that $R$ is a sequence of one-element sets. We will use this notation throughout the paper.

\begin{definition}
\begin{itemize}
\item[(i)]   We will write $I\subseteq R$ if $I=(I_{1},\ldots,I_{t})$ and  $I_{r}\subseteq \{(i_{r},j_{r})\}$ for each $r\in\Sigma$. In the same way, if $I,J\subseteq R$ we say that $J\subseteq I$ if $I=(I_{1},\ldots,I_{t})$, $J=(J_{1},\ldots,J_{t})$ and $J_{r}\subseteq I_{r}$ for any $r\in\Sigma$. 
\item[(ii)] For any $I,J\subseteq R$, we define $I\cup J\subseteq R$, $I\cap J\subseteq R$ by $(I\cup J)_{r}=I_{r}\cup J_{r}$ and $(I\cap J)_{r}=I_{r}\cap J_{r}$, respectively.
\item[(iii)] For each $r \in \Sigma$, denote by $\tau_{r}$ the permutation in $S_{n}\times\cdots\times S_{1}$ that interchanges $i_{r}$ and $j_{r}$ in row $k_{r}$, and that is identity on all other rows. Also, for any $\Delta\subseteq \Sigma$  denote by $\tau_{\Delta}$ the permutation $\tau_{r_{1}}\cdots\tau_{r_{|\Delta|}}$, where $\Delta=\{r_{1},\ldots,r_{|\Delta|}\}$.
\end{itemize}
\end{definition}

The next definition plays central role in the paper.
\begin{definition}
\begin{itemize}
\item[(i)] For any subset $I=(I_{1},\cdots, I_{t})$ of $R$  and for any $z\in  T_{n-1}(\mathbb{Z})$ we introduce  new tableau $\mathcal{D}_{I}T(v+z)$ which we call $I$-\emph{derivative tableau}, or simply \emph{derivative tableau}, and set $\mathcal{D}_{\emptyset}T(v+z) = T(v+z)$. 

\item[(ii)] 
Set $\tilde{V}(T(v))$ to be  the complex vector space  generated by  $\{\mathcal{D}_{I}T(v+z)\; | \; I\subseteq R \text{ and } z\in  T_{n-1}(\mathbb{Z})\}$
subject to the following relations:
\begin{equation}\label{relations satisfied by vectors on this universal module}
\mathcal{D}_{I}T(v+\tau_{r}(z))=
\begin{cases}
\mathcal{D}_{I}T(v+z), & \text{ if }\ \ \ \ \ \  I_{r}=\emptyset\\
-\mathcal{D}_{I}T(v+z), & \text{ if }\ \ \ \ \ \  I_{r}\neq\emptyset.
\end{cases}
\end{equation}
\end{itemize}
\end{definition}

\begin{remark}\label{rem-basis}
Although the spanning set $\{\mathcal{D}_{I}T(v+z)\; | \; I\subseteq R \text{ and } z\in  T_{n-1}(\mathbb{Z})\}$ is not a basis $\tilde{V}(T(v))$, it will be  convenient to work with the whole spanning set and then verify the relations separately. A basis of $\tilde{V}(T(v))$ is, for instance, the set of all $\mathcal{D}_{I}T(v+z)$ such that $z_{k_{r}i_{r}}-z_{k_{r}j_{r}}> 0$ if $I_{r} \neq\emptyset$, and $z_{k_{r}i_{r}}-z_{k_{r}j_{r}} \leq  0$ if $I_{r} = \emptyset$.
\end{remark}

\begin{proposition}
There is a natural isomorphism between the spaces ${V}(T(v))$ and $\tilde{V}(T(v))$. \end{proposition}

\begin{proof}
Let us fix the basis of $\tilde{V}(T(v))$ defined in Remark \ref{rem-basis}. Let $z\in  T_{n-1}(\mathbb{Z})$ and $T(v+z)\in V(T(v))$. Consider any $\Delta=\{r_{1},\ldots,r_{|\Delta|}\}\subseteq \Sigma$ and $T(v+\tau_{\Delta}(z))\in V(T(v))$, where $\tau_{\Delta}=\tau_{r_{1}}\cdots\tau_{r_{|\Delta|}}$.  Recall that $T(v+z)$ and $T(v+\tau_{\Delta}(z))$ define the same maximal ideal of  $\Gamma$. To complete the proof, we need to identify the tableau $T(v+\tau_{\Delta}(z))\in V(T(v))$ with a derivative tableau in $\tilde{V}(T(v))$. For each $r\in\Sigma$ set $I_{r}= \{(i_{r},j_{r})\}$ if  $r\in\Delta$ and $ I_{r}=\emptyset$ otherwise. Let
$I=(I_{1},\ldots,I_{t})$. Then $\mathcal{D}_{I}T(v+z)$ is the derivative tableau corresponding to $T(v+\tau_{\Delta}(z))$ that we need. Clearly, this identification extends to a linear isomorphism between 
${V}(T(v))$ and $\tilde{V}(T(v))$. 
\end{proof}

From now on we will identify the space ${V}(T(v))$ and $\tilde{V}(T(v))$ and the rest of this section, as well as the next section, are devoted to defining an appropriate $\mathfrak{gl}(n)$-action on that space.

Denote by $T_{n}(\mathbb{C})_{\rm reg}$ the set of all regular vectors in $T_{n}(\mathbb{C})$. 
Recall that if $x \in T_{n}(\mathbb{C})_{\rm gen}$ then $V(T(v))$ is irreducible if and only if $x \in T_{n}(\mathbb{C})_{\rm gen}\cap T_{n}(\mathbb{C})_{\rm reg}$, see Theorem 6.14 in  \cite{FGR1}.

Let $\mathcal{S}^{0}$ be the set of vectors $x$ in $T_{n}(\mathbb{C})_{\rm gen}\cap T_{n}(\mathbb{C})_{\rm reg}$ such that $0\leq Re(x_{rs}-x_{r-1,s})<1$ for any $r,s$. 
\begin{lemma}
$\mathcal{S}^{0}+ T_{n-1}(\mathbb{Z})=T_{n}(\mathbb{C})_{\rm gen}\cap T_{n}(\mathbb{C})_{\rm reg}$. Moreover, for any $w\neq w'$ in $\mathcal{S}^{0}$ we have $\left(w+ T_{n-1}(\mathbb{Z})\right)\cap \left(w'+ T_{n-1}(\mathbb{Z})\right)=\emptyset$.

\end{lemma}
\begin{proof}
For any $w\in T_{n}(\mathbb{C})_{\rm gen}\cap T_{n}(\mathbb{C})_{\rm reg}$ let $x\in w +  T_{n-1}(\mathbb{Z}) $ be the vector for which:
$$
x_{rs}=\begin{cases}
w_{rs}+\sum_{j=r+1}^{n}\lfloor Re(w_{j,s}-w_{j-1,s})\rfloor, & \text{ if }\ \ r\leq n-1\\
w_{ns}, & \text{ if }\ \ r=n.
\end{cases}
$$
We have $x\in \mathcal{S}^{0}$. Indeed, for any $r,s$, 
$$x_{rs}-x_{r-1,s}=
w_{rs}-w_{r-1,s}-\lfloor Re(w_{r,s}-w_{r-1,s})\rfloor,
$$ which implies $0\leq Re(x_{rs}-x_{r-1,s})<1$. Hence, $w \in \mathcal{S}^{0} +  T_{n-1}(\mathbb{Z}) $
 and $\mathcal{S}^{0}+ T_{n-1}(\mathbb{Z})=T_{n}(\mathbb{C})_{\rm gen}\cap T_{n}(\mathbb{C})_{\rm reg}$. For the second part of the lemma it is enough to prove that $\left(w+ T_{n-1}(\mathbb{Z})\right)\cap \mathcal{S}^{0}=\{w\}$. Let $z\in  T_{n-1}(\mathbb{Z})$ be  such that $w+z\in\mathcal{S}^{0}$. Then the conditions $0\leq Re((w+z)_{rs}-(w+z)_{r-1,s})<1$ and $0\leq Re(w_{rs}-w_{r-1,s})<1$ imply $z_{rs}=z_{r-1,s}$. In particular, $z_{ns}=z_{r,s}=0$ for any $r\leq n-1$ and hence $z=0$. 
\end{proof}

In view of the last lemma we introduce  ${\mathcal V}_{\rm gen} := \bigoplus_{x \in \mathcal{S}^0} V(T(x))$. By Proposition \ref{coefficients e_ij}, ${\mathcal V}_{\rm gen}$ is a Gelfand-Tsetlin module. We call this module  the \emph{family of generic Gelfand-Tsetlin modules}. Note that ${\mathcal V}_{\rm gen} = \bigoplus_{x \in T_{n}(\mathbb{C})_{\rm gen}\cap T_{n}(\mathbb{C})_{\rm reg}} {\mathbb C}T(v)$ as vector spaces.

\begin{definition} Let $\Delta \subseteq \Sigma$ and $I\subseteq R$.
\begin{itemize}
\item[(i)] Define $R_{\Delta}$ to be the subset of $R$ whose $r$-th component is
 $$(R_{\Delta})_{r}:=\begin{cases}
\{(i_{r},j_{r})\}, & \text{ if }\ \ \ \ \ \  r\in \Delta\\
\emptyset, & \text{ if }\ \ \ \ \ \  r\notin \Delta.\end{cases}$$
We also write
$$ P_{\Delta}(x):=\prod_{r\in \Delta}(x_{k_{r},i_{r}}-x_{k_{r},j_{r}}).$$ 
\item[(ii)] Define $\Sigma_{I}:=\{r\in\Sigma\; | \; I_{r}\neq \emptyset\}\subseteq \Sigma$. 
\end{itemize}
\end{definition}
\noindent
From the above definition we  easily obtain that  $\Sigma_{R}=\Sigma$ and $R_{\Sigma}=R$. Also, for any $I,J\subseteq R$ and $\Delta_{1},\Delta_{2}\subseteq\Sigma$ we have $\Sigma_{I}\cup \Sigma_{J}=\Sigma_{I\cup J}$ and $R_{\Delta_{1}}\cup R_{\Delta_{2}}=R_{\Delta_{1}\cup \Delta_{2}}$.\\

 Denote by ${\mathcal H_{ij}^{k}}\subseteq T_n (\mathbb C)$ the hyperplane $x_{k,i} - x_{kj} = 0$. We also set $\mathcal{H}=\bigcap_{r\in\Sigma}\mathcal{H}_{i_{r}j_{r}}^{k_{r}}$ and $\overline{\mathcal{H}}=\bigcap_{(k,i,j)\neq (k_{r},i_{r},j_{r})}(\mathcal{H}_{ij}^{k})^{c}$, where $A^c$ stands for the complement of $A$ in $T_n (\mathbb C)$ . In other words, 
${\mathcal H}\subseteq T_n (\mathbb C)$  is the intersection of the critical hyperplanes $x_{k_{r},i_{r}} - x_{k_{r}j_{r}} = 0$, $r\in\Sigma$, while $\overline{\mathcal H}$ consists of all $x$ in $T_n (\mathbb C) $ such that $x_{ki} \neq x_{kj}$ for all triples $(k,i,j)$ except for $(k,i,j) = (k_{r},i_{r},j_{r})$, $r\in\Sigma $. 

Denote by ${\mathcal F}$ the space of rational functions in $x_{k\ell}$, $1\leq \ell \leq k \leq n$, with poles only on the union of the hyperplanes $\mathcal{H}_{i_{r}j_{r}}^{k_{r}}$, $r\in\Sigma$. Let $\overline{{\mathcal F}}$ be  the subspace of ${\mathcal F}$ consisting of all those functions that are smooth on $\overline{\mathcal H}$. Finally, we will say that $f\in\mathcal{F}$ is a {\it smooth function} if $f\in\overline{\mathcal{F}}$.

Recall that  $v$ is a fixed element in  ${\mathcal H} \cap \overline{\mathcal{H}}$. 
In order to introduce the operator $\mathcal{D}^{v}_{I}$  on $\overline{\mathcal{F}} \otimes {\mathcal V}_{\rm gen}$, 
we first define the operators $\mathcal{D}_{I}$ and $\mathcal{D}^{v}_{I}$ on $\overline{\mathcal{F}}$.

\begin{definition} \label{def-d-i}
For any subset $I=(I_{1},\cdots, I_{t})$ of $R$  and for any $z\in  T_{n-1}(\mathbb{Z})$ we define differential operators $\mathcal{D}_{I} : \overline{\mathcal{F}} \to \overline{\mathcal{F}}$, $\mathcal{D}_{I}^{v} : \overline{\mathcal{F}} \to \mathbb C$ as follows. For a smooth function $f$ (i.e. $f \in \overline{\mathcal{F}}$),
$$\mathcal{D}_{I} (f)= \mathcal{D}_{I_{t}}(\cdots\mathcal{D}_{I_{2}}(\mathcal{D}_{I_{1}}(f))\cdots), \ \ \  \mathcal{D}_{I}^{v}(f)= \mathcal{D}_{I} (f) (v),$$
where 
$$\mathcal{D}_{I_{r}}(g)=\begin{cases}
g, & \text{ if } \ \ \ \ \ \ I_{r}=\emptyset\\
\frac{1}{2}\left(\frac{\partial g}{\partial x_{k_{r}i_{r}}}-\frac{\partial g}{\partial x_{k_{r}j_{r}}}\right), & \text{ if } \ \ \ \ \ \ I_{r}\neq\emptyset.\\
\end{cases}$$ In particular, for any  smooth functions $f,g$ we have:
$$\mathcal{D}_{I}^{v}(fg):=\sum\limits_{J\subseteq I}\mathcal{D}^{v}_{I\setminus J}(f)\mathcal{D}^{v}_{J}(g).$$ 
\end{definition}

\begin{definition} 
 We define the operators $\mathcal{D}_{I}^{v} : \overline{\mathcal{F}} \otimes {\mathcal V}_{\rm gen} \to \tilde{V}(T(v))$ as the linear maps for which
$$
\mathcal{D}_{I}^{v}(fT(x+z)):=\sum\limits_{J\subseteq I}\left(\mathcal{D}^{v}_{I\setminus J}(f)\mathcal{D}_{J}T(v+z)\right),$$  for $x\in \mathcal{S}^{0}$.
In particular, $\mathcal{D}_{I}^{v} (T(x+z)) = \mathcal{D}_{I} (T(v+z))$.
We set $\mbox{ev}(v) = \mathcal{D}_{\emptyset}^{v}$ and call it the \emph{evaluation map} on $\overline{\mathcal{F}} \otimes {\mathcal V}_{\rm gen}$.

\end{definition}

\begin{remark}
Note that $2^{|\Sigma_{I}|}\mathcal{D}_{I}=P_{\Sigma_{I}}(\partial_{x_{n1}},\ldots,\partial_{x_{nn}},\ldots,\partial_{x_{21}},\partial_{x_{22}},\partial_{x_{11}})$ as operators on $\overline{{\mathcal F}}$.
\end{remark}

\begin{definition} Given $x\in\mathcal{S}^{0}$, $v\in T_n(\mathbb C)$ and $\Delta\subsetneq \Sigma$, we define an element $v_{\Delta}\in T_{n}(\mathbb{C})$  whose $(ij)$-th component is
$$
(v_{\Delta})_{ij}=
\begin{cases}
x_{ij}, &  \text{ if }\ \ \ \ \ \  (i,j)\notin\{(k_{r},i_{r}),(k_{r},j_{r}) \; | \; r \in \Delta \}\\
v_{ij}, & \text{ if }\ \ \ \ \ \  (i,j)\in\{(k_{r},i_{r}),(k_{r},j_{r})\; | \; r  \in \Delta \}
\end{cases}
$$
Also, set $v_{\Sigma}:=v$. Note that $v_{\emptyset}=x$  and that $v_{\Delta}$ is  $|\Delta|$-singular of index $2$ if $\Delta\neq \emptyset$.
\end{definition}

\begin{remark}\label{evaluation step by step}
By condition $(iii)$ in Definition \ref{Def of 1^t singular}, we have 
$$
\mathcal{D}_{I}^{v}(f)=ev(v)\mathcal{D}_{I_{t}}(\cdots\mathcal{D}_{I_{2}}(\mathcal{D}_{I_{1}}(f))\cdots)
=\mathcal{D}^{v_{\Sigma_{t}}}_{I_{\sigma(t)}}\left(\mathcal{D}^{v_{\Sigma_{t-1}}}_{I_{\sigma(t-1)}}\left(\cdots\left(\mathcal{D}^{v_{\Sigma_{1}}}_{I_{\sigma(1)}}(f)\right)\cdots\right)\right)
$$
where $\sigma$ is any permutation in $S_{t}$ and $\Sigma_{i}=\{\sigma(1),\ldots,\sigma(i)\}$. In the above identity, $\mathcal{D}^{v_{\Sigma_{1}}}_{I_{\sigma(1)}}(f)$ is treated as a function in $2t-2$ variables, $\mathcal{D}^{v_{\Sigma_{2}}}_{I_{\sigma(2)}}\mathcal{D}^{v_{\Sigma_{1}}}_{I_{\sigma(1)}}(f)$ is treated as a function in $2t-4$ variables, and so forth.
\end{remark}

The following lemma list some useful properties of the operators $\mathcal{D}_{I}$ that will later be formulated and proved for the corresponding derivative tableaux.

\begin{lemma}\label{derivatives acting on products}
For any smooth function $f$, any $I\subseteq R$ and any $\Delta\subseteq \Sigma$ we have:

$$\mathcal{D}^{v}_{I}\left(P_{\Delta}(x)f\right)=
\begin{cases}
\mathcal{D}^{v}_{I\setminus R_{\Delta}}(f), & \text{ if } R_{\Delta}\subseteq I\\
0, & \text{ if } R_{\Delta}\nsubseteq I.
\end{cases}$$
In particular, $\mathcal{D}^{v}_{R_{\Delta}}\left(P_{\Delta}(x)f\right)=f(v).$
\end{lemma}
\begin{proof} 
By the definition of $\mathcal{D}^{v}_{I}$ we have $\mathcal{D}^{v}_{I}\left(P_{\Delta}(x)f\right)=\sum\limits_{J\subseteq I}\mathcal{D}^{v}_{I\setminus J}(f)\mathcal{D}^{v}_{J}(P_{\Delta}(x))$. Therefore,
$$\mathcal{D}^{v}_{J}(P_{\Delta}(x))=
\begin{cases}
1, & \text{ if }\ \  R_{\Delta}=J\\
0, & \text{ if }\ \  R_{\Delta}\neq J.
\end{cases}
$$
\end{proof}

\subsection{Identities for divided differences of rational functions}\label{subsection: Identities divided differences}
We fix $z\in T_{n-1}(\mathbb{Z})$ and consider $x\in\mathcal{S}^{0}$ as a variable. For any $\Delta\subset \Sigma$ and any rational function $f\in\mathbb{C}(x_{ij}\ |\ ;1\leq j\leq i\leq n)$, by $f(x)^{\tau_{\Delta}}$ we denote the corresponding $\tau_{\Delta}$-twisted function, i.e. $f(x)^{\tau_{\Delta}}=f(\tau_{\Delta}(x))$. In this section we deal extensively with functions of $x+z$, so, for convenience we set $y=x+z$.
By default, for any rational function $f$, $f(y)$ will stand for the function $f(x+z)$. In particular, we have $f(y)^{\tau_{\Delta}}=f(\tau_{\Delta}(x)+z)$.
\begin{definition} Let $f$ be a rational function and $\Delta\subseteq\Sigma$.
We define \emph{the $\Delta$-divided difference of $f$ at $x$}, $f(x)_{\Delta}$, as follows 
$$f(x)_{\Delta}:=\frac{1}{P_{\Delta}(x)}\sum\limits_{\bar{\Delta}\subseteq\Delta}(-1)^{|\bar{\Delta}|}f(x)^{\tau_{\bar{\Delta}}}.$$
 In particular, we write:
$$f(y)_{\Delta}=\sum\limits_{\bar{\Delta}\subseteq\Delta}\frac{(-1)^{|\bar{\Delta}|}f(y)^{\tau_{\bar{\Delta}}}}{P_{\Delta}(x)}.$$
\end{definition}

\begin{remark} We often consider divided differences of products $f(x)g(y)$ of functions of $x$ and $y=x+z$. In such a case, one should keep in mind that:
\begin{align*} (f(x)g(y))_{\Delta}&=\sum\limits_{\bar{\Delta}\subseteq\Delta}\frac{(-1)^{|\bar{\Delta}|}(f(x)g(y))^{\tau_{\bar{\Delta}}}}{P_{\Delta}(x)}\\
&=\sum\limits_{\bar{\Delta}\subseteq\Delta}\frac{(-1)^{|\bar{\Delta}|}f(\tau_{\bar{\Delta}}(x))g(\tau_{\bar{\Delta}}(x)+z)}{P_{\Delta}(x)}
\end{align*}
\end{remark}

\begin{lemma}\label{some identities} Let $I$ be any subset of $R$, $f$ be a rational function, and $\Delta\subseteq\Sigma$. 
\begin{itemize}
\item[(i)] If $f$ is smooth and $f(y)=f(y)^{\tau_{r}}$ for some $r\in\Sigma_{I}$, then $\mathcal{D}^{v}_{I}(f(y))=0$.
\item[(ii)] For any $s\in\Delta$ we have  $f(y)_{\Delta}=(f(y)_{\Delta\setminus\{s\}})_{\{s\}}$. In particular, $f(y)_{\Delta}$ is $\tau_{s}$-invariant for any $s\in\Delta$, i.e. $(f(y)_{\Delta})^{\tau_{s}}=f(y)_{\Delta}$.
\item[(iii)] If $f(y)$ is smooth then $f(y)_{\Delta}$ is smooth and $$\mathcal{D}^{v}_{I}(f(y)_{\Delta})=\begin{cases}
2^{|\Delta|}\mathcal{D}^{v}_{I\cup R_{\Delta}}(f(y)),& \text{ if } \Delta\subseteq\Sigma\setminus\Sigma_{I}\\
0,& \text{ if } \Delta\nsubseteq\Sigma\setminus\Sigma_{I}
\end{cases}.$$ 
\end{itemize}
\end{lemma}
\begin{proof} For all three parts we use crucially that the lemma holds in the case $t=1$ (see Lemma A.1 in \cite{FGR2}).
\begin{itemize}
\item[(i)] Let  $r\in\Sigma_{I}$ be such that $f(y)=f(y)^{\tau_{r}}$. Then for  any permutation $\sigma$ in $S_{t}$ such that $\sigma(1)=r$,
$$\mathcal{D}^{v}_{I}(f(y))= \mathcal{D}^{v_{\Sigma_{t}}}_{I_{\sigma(t)}}\left(\mathcal{D}^{v_{\Sigma_{t-1}}}_{I_{\sigma(t-1)}}\left(\cdots\left(\mathcal{D}^{v_{\Sigma_{1}}}_{I_{\sigma(1)}}(f(y))\right)\cdots\right)\right)
=0,$$
where we used that $\mathcal{D}^{v_{\Sigma_{1}}}_{I_{\sigma(1)}}(f(y))=0$ by Lemma A.1(i) in \cite{FGR2}.
\item[(ii)] This part follows by a straightforward verification. Namely, one  checks that: $$f(y)_{\Delta}=\frac{f(y)_{\Delta\setminus\{s\}}-(f(y)_{\Delta\setminus\{s\}})^{\tau_{s}}}{P_{\{s\}}(x)}.$$ 
\item[(iii)] Let first  $\Delta\subseteq \Sigma\setminus\Sigma_{I}$  and let $s\in\Delta$. In particular $I_{s}=\emptyset$, and then using part (ii) and Lemma A.1(ii) in \cite{FGR2}, we have 
$$\mathcal{D}^{v_{\{s\}}}_{I_{s}}(f(y)_{\Delta})=\mathcal{D}^{v_{\{s\}}}_{I_{s}}((f(y)_{\Delta\setminus\{s\}})_{\{s\}})=2\mathcal{D}^{v_{\{s\}}}_{R_{\{s\}}}(f(y)_{\Delta\setminus\{s\}})$$
Now, for any permutation $\sigma$ in $S_{t}$, we obtain:
\begin{align*}
\mathcal{D}^{v}_{I}(f(y)_{\Delta})=& \mathcal{D}^{v_{\Sigma_{t}}}_{I_{\sigma(t)}}\left(\mathcal{D}^{v_{\Sigma_{t-1}}}_{I_{\sigma(t-1)}}\left(\cdots\left(\mathcal{D}^{v_{\Sigma_{1}}}_{I_{\sigma(1)}}(f(y)_{\Delta})\right)\cdots\right)\right)\\
=& 2^{|\Delta|}\mathcal{D}^{v}_{I\cup R_{\Delta}}(f(y)),
\end{align*}
where if $\Sigma_{i}=\{\sigma(1),\ldots,\sigma(i)\}\subseteq \Sigma$.

Finally, if $\Delta\nsubseteq\Sigma\setminus\Sigma_{I}$, then taking $r\in\Delta\cap\Sigma_{I}$, by party (ii), $f(y)_{\Delta}$ is $\tau_{r}$-invariant and then  part (i) implies $\mathcal{D}^{v}_{I}(f(y)_{\Delta})=0$.
\end{itemize}
\end{proof}

\begin{lemma}\label{previous of lemma 8.2 generalized} 
Let $\Delta_{0},\Delta\subseteq\Sigma$ be fixed and $f_m(y), g_m(y)$, $m=1, \ldots, s$, be rational functions such that for any $\bar{\Delta}\subsetneq\Delta$ and $r\in\Delta\setminus\bar{\Delta}$ we have $\sum_{m=1}^s f_m(y) (g_m(y)_{\bar{\Delta}})^{\tau_{r}}=0$. Then
$$P_{\Delta_{0}}(x)\sum_{m=1}^{s}f_{m}(y)g_{m}(y)=P_{\Delta_{0}\setminus\Delta}(x)\sum_{m=1}^{s} f_{m}(y)(P_{\Delta}(x)P_{\Delta\cap\Delta_{0}}(x)g_{m}(y))_{\Delta}$$

\end{lemma}
\begin{proof} We prove the identity in three steps.

\noindent {\it Step 1.} We prove that for any $\Delta_{1}\subseteq \Delta\cap\Delta_{0}$ we have:
\begin{equation}\label{eqn omega}
P_{\Delta_{0}}(x)\sum_{m=1}^{s}f_{m}(y)g_{m}(y)=P_{\Delta_{0}\setminus\Delta_{1}} (x) \sum_{m=1}^{s} f_{m}(y)(P_{\Delta_{1}}^{2}(x)g_{m}(y))_{\Delta_{1}}.
\end{equation}
By taking $\bar{\Delta}=\emptyset$ in our hypothesis we see that for any $r\in\Delta\cap\Delta_{0}$,  $\sum\limits_{m=1}^{s}f_{m}(y)(g_{m}(y))^{\tau_{r}}=0$. Thus:
$$
P_{\Delta_{0}}(x)\sum_{m=1}^{s}f_{m}(y)g_{m}(y)=P_{\Delta_{0}\setminus\{r\}}(x)\sum_{m=1}^{s}f_{m}(y)(P^{2}_{\{r\}}(x)g_{m}(y))_{\{r\}}.
$$
To prove (\ref{eqn omega}),  we apply induction on $|\Delta_1|$.\\
\medskip
{\it Step 2.} We prove that for any $\Delta_{2}\subseteq\Delta\setminus\Delta_{0}$, we have:
\begin{equation}\label{eqn omega2}
\sum_{m=1}^{s} f_{m}(y)(g_{m}(y))_{\Delta\cap\Delta_{0}}=\\\sum_{m=1}^{s} f_{m}(y)(P_{\Delta_{2}}(x)g_{m}(y))_{(\Delta\cap\Delta_{0})\cup\Delta_{2}}.
\end{equation}
By hypothesis, for any $r\in\Delta\setminus\Delta_{0}$, we have 
$\sum_{m=1}^{s} f_{m}(y)(g_{m}(y)_{\Delta\cap\Delta_{0}})^{\tau_{r}}=0$. Thus:
$$
\sum_{m=1}^{s} f_{m}(y)(g_{m}(y))_{\Delta\cap\Delta_{0}}=\sum_{m=1}^{s} f_{m}(y)(P_{\{r\}}(x)g_{m}(y))_{(\Delta\cap\Delta_{0})\cup\{r\}}.$$
To prove (\ref{eqn omega2}) we proceed by induction on $|\Delta _2|$.\\
{\it Step 3.} We apply Steps 2 and 3 for  $\Delta_{1}=\Delta\cap\Delta_{0}$ and $\Delta_{2}=\Delta\setminus\Delta_{0}$.

More precisely, from (\ref{eqn omega}), we have:
\begin{align*}
P_{\Delta_{0}}(x)\sum_{m=1}^{s}f_{m}(y)g_{m}(y)
&=P_{\Delta_{0}\setminus\Delta}(x)\sum_{m=1}^{s} f_{m}(y)(P_{\Delta\cap\Delta_{0}}^{2}(x)g_{m}(y))_{\Delta\cap\Delta_{0}}\\
&=P_{\Delta_{0}\setminus\Delta}(x)P_{\Delta\cap\Delta_{0}}^{2}(x)\sum_{m=1}^{s} f_{m}(y)(g_{m}(y))_{\Delta\cap\Delta_{0}}.
\end{align*}
On the other hand,  (\ref{eqn omega2}) implies:
\begin{align*}
\sum_{m=1}^{s} f_{m}(y)(g_{m}(y))_{\Delta\cap\Delta_{0}}
&=\sum_{m=1}^{s} f_{m}(y)(P_{\Delta\setminus\Delta_{0}}(x)g_{m}(y))_{(\Delta\cap\Delta_{0})\cup(\Delta\setminus\Delta_{0})}\\
&=\sum_{m=1}^{s} f_{m}(y)(P_{\Delta\setminus\Delta_{0}}(x)g_{m}(y))_{\Delta}.
\end{align*}
Therefore, 
\begin{align*}
P_{\Delta\cap\Delta_{0}}^{2}(x)\sum_{m=1}^{s} f_{m}(y)(g_{m}(y))_{\Delta\cap\Delta_{0}}&=\sum_{m=1}^{s} f_{m}(y)(P_{\Delta\cap\Delta_{0}}^{2}(x)P_{\Delta\setminus\Delta_{0}}(x)g_{m}(y))_{\Delta}\\
&=\sum_{m=1}^{s} f_{m}(y)(P_{\Delta}(x)P_{\Delta\cap\Delta_{0}}(x)g_{m}(y))_{\Delta}.
\end{align*}
\end{proof}
The following lemma gives sufficient conditions for the functions $f_{m}(y),g_{m}(y)$ to satisfy the identity $\sum_{m=1}^s f_m (y) (g_{m} (y)_{\bar{\Delta}})^{\tau_{r}}=0$.

\begin{lemma}\label{release conditions}
Let $f_m(y), g_m(y)$, $m=1, \ldots, s$, be any set of rational functions and $\Delta\subseteq\Sigma$. If $\sum_{m=1}^s f_m (y) (g_m(y))^{\tau_{\Delta_{1}}}=0$ for any $\emptyset\neq\Delta_{1}\subseteq\Delta$, then for any $\bar{\Delta}\subsetneq\Delta$ and $r\in\Delta\setminus\bar{\Delta}$ we have $\sum_{m=1}^s f_m(y) (g_m(y)_{\bar{\Delta}})^{\tau_{r}}=0$.
\end{lemma}
\begin{proof}
The statement follows directly from the definition of $(g_{m}(y)_{\bar{\Delta}})^{\tau_{r}}$. In fact, $$(g_{m}(y)_{\bar{\Delta}})^{\tau_{r}}=\sum\limits_{\Delta'\subseteq\bar{\Delta}}\frac{(-1)^{|\Delta'|}g_{m}(y)^{\tau_{\Delta'\cup\{r\}}}}{P_{\bar{\Delta}}(x)}$$ and for any $\Delta'\subsetneq\bar{\Delta}$, the set $\Delta_{1}=\Delta'\cup\{r\}$ is a nonempty subset of $\Delta$. 
\end{proof}

\begin{proposition}\label{Most general version of lemma 8.2}
Let $f_m(y), g_m(y)$, $m=1, \ldots, s$, be rational functions and let $\Delta\subseteq\Sigma$. Assume that $f_m(y)$, $P_{\Delta}(x)g_m(y)$, and $\sum_{m=1}^s f_m(y) g_m(y)$ are smooth functions, and also that for any $\emptyset\neq\bar{\Delta}\subseteq\Delta$ we have $\sum_{m=1}^s f_m(y) g_m(y)^{\tau_{\bar{\Delta}}}=0$. Then the following identity holds:
$$\mathcal{D}^{v}_{I}\left(\sum_{m=1}^{s}f_{m}(y)g_{m}(y)\right)=\sum_{m=1}^{s}\sum_{\substack{J\subseteq I\\J\cap R_{\Delta}=\emptyset}}2^{|\Delta|}\mathcal{D}^{v}_{(I\setminus J)\cup R_{\Delta}}(f_{m}(y))\mathcal{D}^{v}_{J\cup (R_{\Delta}\cap I)}(P_{\Delta}(x)g_{m}(y)).$$
\end{proposition}

\begin{proof}
First, note that by Lemma \ref{release conditions}, the hypothesis of Lemma \ref{previous of lemma 8.2 generalized} is satisfied. Furthermore, 
$$
\mathcal{D}^{v}_{I}\left(\sum\limits_{m=1}^{s}f_{m}(y)g_{m}(y)\right) =\mathcal{D}^{v}_{R}\left(P_{\Sigma\setminus\Sigma_{I}}(x)\sum\limits_{m=1}^{s}f_{m}(y)g_{m}(y)\right)$$
$$=\mathcal{D}^{v}_{R}\left(\sum_{m=1}^{s}(P_{(\Sigma\setminus\Sigma_{I})\setminus\Delta}(x) f_{m}(y))(P_{\Delta}(x)P_{\Delta\cap(\Sigma\setminus\Sigma_{I})}(x)g_{m}(y))_{\Delta}\right) $$
$$=\sum_{m=1}^{s}\sum_{J\subseteq R}\mathcal{D}^{v}_{R\setminus J}\left(P_{(\Sigma\setminus\Sigma_{I})\setminus\Delta}(x) f_{m}(y)\right)\mathcal{D}^{v}_{J}((P_{\Delta}(x)P_{\Delta\cap(\Sigma\setminus\Sigma_{I})}(x)g_{m}(y))_{\Delta}) $$
$$=\sum_{m=1}^{s}\sum_{\substack{J\subseteq R\\J\cap R_{\Delta}=\emptyset}}\mathcal{D}^{v}_{R\setminus J}\left(P_{(\Sigma\setminus\Sigma_{I})\setminus\Delta}(x) f_{m}(y)\right)\mathcal{D}^{v}_{J}((P_{\Delta}(x)P_{\Delta\cap(\Sigma\setminus\Sigma_{I})}(x)g_{m}(y))_{\Delta}).$$
The second equality follows from Lemma \ref{previous of lemma 8.2 generalized}(iii), while, the last equality follows from Lemma \ref{some identities}(i) and the fact that $(P_{\Delta}(x)P_{\Delta\cap(\Sigma\setminus\Sigma_{I})}(x)g_{m}(y))_{\Delta}$ is $\tau_{r}$-invariant for any $r\in\Delta$. Also, by Lemma \ref{some identities}(iii) we have
$$\mathcal{D}^{v}_{J}((P_{\Delta}(x)P_{\Delta\cap(\Sigma\setminus\Sigma_{I})}(x)g_{m}(y))_{\Delta})=2^{|\Delta|}\mathcal{D}^{v}_{J\cup R_{\Delta}}(P_{\Delta}(x)P_{\Delta\cap(\Sigma\setminus\Sigma_{I})}(x)g_{m}(y)).$$

Finally, since  $f_{m}(y)$ and $P_{\Delta}(x)g_{m}(y)$ are smooth functions, by Lemma \ref{derivatives acting on products} we have 

$$\mathcal{D}^{v}_{R\setminus J}\left(P_{(\Sigma\setminus\Sigma_{I})\setminus\Delta}(x) f_{m}(y)\right)=\begin{cases}
\mathcal{D}^{v}_{(I\setminus J)\cup R_{\Delta}}( f_{m}(y)),& \text{ if } J\subseteq I\\
0,& \text{ if } J\nsubseteq I
\end{cases}$$ and 
$$\mathcal{D}^{v}_{J\cup R_{\Delta}}((P_{\Delta}(x)P_{\Delta\cap(\Sigma\setminus\Sigma_{I})}(x)g_{m}(y))=\mathcal{D}^{v}_{J\cup (R_{\Delta}\cap I)}(P_{\Delta}(x)g_{m}(y)).$$
\end{proof}

\subsection{Identities for $e_{k\ell}(x)$ and $\varepsilon_{k\ell}$}\label{subsection: Identities for e_rs}
In this section we prove some useful identities for the functions $e_{k\ell}(x)$ and $\varepsilon_{k\ell}$ defined in  Definition \ref{definition of coefficients e_rs}.

Recall the definition of $\Phi_{k\ell} $ in \S \ref{sec-gen-gt}.  For $\min \{ \ell,k\} \leq u \leq \max \{ \ell,k\} -1$ and $1 \leq s \leq k$ we set $
\Phi_{k\ell} (u,s) = \{ \sigma \in \Phi_{k\ell} \; | \; \sigma[u] = (1,s)\}.
$
For the rest of the appendix  we will need  $\Phi_{k\ell} (u,s)$ mostly for $u=k_{r}$, and $s=i_{r}$ or $s =j_{r}$. We set for convenience $\Phi_{k\ell} (i_{r}) = \Phi_{k\ell} (k_{r},i_{r})$ and $\Phi_{k\ell} (j_{r})  = \Phi_{k\ell} (k_{r},j_{r})$.

\begin{definition}
For each $r\in\Sigma$ and $\sigma\in\Phi_{k\ell}$ we define:
$$\tau_{r}^\star(\sigma):=
\begin{cases}
\sigma, & \text{ if } \ \ \ \sigma\notin\Phi_{k\ell} (i_{r})  \cup \Phi_{k\ell}(j_{r});\\
\tau_{r}\sigma\tau_{r}=\sigma\tau_{r}\sigma, & \text{ if } \ \ \ \sigma\in\Phi_{k\ell} (i_{r})  \cup \Phi_{k\ell}(j_{r})\text{ and } 1\notin\{i_{r},j_{r}\};\\
\tau_{r}\sigma=\sigma\tau_{r}, & \text { if } \ \ \ \sigma\in\Phi_{k\ell} (i_{r})  \cup \Phi_{k\ell}(j_{r}) \text{ and } 1\in\{i_{r},j_{r}\}.
\end{cases}$$
 For any subset $\Delta = \{r_{1},\ldots,r_{|\Delta|}\}$ of $\Sigma$,  by $\tau^{\star}_{\Delta}$ we denote the operator on $\Phi_{k\ell}$ defined by $\tau^{\star}_{\Delta}(\sigma)=\tau_{r_{1}}^\star(\cdots(\tau^{\star}_{r_{|\Delta|}}(\sigma)))$.
\end{definition}

\begin{remark} 
One easily shows that $\tau_{r}^\star(\sigma)\in \Phi_{k\ell}$, and that if $\sigma\in\Phi_{k\ell} (i_{r})  \cup \Phi_{k\ell}(j_{r})$, then $\tau_{r}^\star(\sigma)\in\Phi_{k\ell} (i_{r})  \cup \Phi_{k\ell}(j_{r})$ and $\tau_{r}^\star(\sigma) \neq \sigma$. Also, note that $\tau^{\star}_{\Delta}$ is an operator on $\tilde{S}_{n-1}\times\cdots\times\tilde{S}_{1}$ that acts as identity on $\tilde{S}_{i}$ if $i\notin\{k_{r}\; | \; r\in\Delta\}$ and that interchanges the transpositions $(1,i_{r})$ and $(1,j_{r})$ of $\tilde{S}_{k_{r}}$, for all $r\in\Delta$.
\end{remark}

\begin{lemma}\label{properties of e_rs and epsilon_rs with star product}
Let $w'\in T_{n}(\mathbb{C})$, $\sigma\in\tilde{S}_{n-1}\times\cdots\times\tilde{S}_{1}$, and $\Delta\subseteq\Sigma$. Then the following identities hold.
\begin{itemize}
\item[(i)]
$e_{k\ell}(\tau^{\star}_{\Delta}(\sigma) (w'))=e_{k\ell}(\sigma\tau_{\Delta}(w'))$.
\item[(ii)] $\tau^{\star}_{\Delta}(\sigma)(\varepsilon_{k\ell})=\tau_{\Delta}\sigma(\varepsilon_{k\ell}).$
\end{itemize}
\end{lemma}
\begin{proof}
The lemma follows by a straightforward verification.
\end{proof}

\begin{lemma}\label{Generalization of lemma 8.7}
Let $I \subseteq R$ and $\Delta \subseteq \Sigma$ be such that $P_{\Delta}(v)e_{k\ell}(\tau_{r}^\star(\sigma)(x+\tau_{r}(z))$ and $P_{\Delta}(x)e_{k\ell}(\sigma(x+z))$ are smooth functions. Then:
$$\mathcal{D}_{I}^{v}\left(P_{\Delta}(v)e_{k\ell}(\tau_{r}^\star(\sigma)(x+\tau_{r}(z)))\right)=
\begin{cases}
\mathcal{D}_{I}^{v}\left(P_{\Delta}(x)e_{k\ell}(\sigma(x+z))\right), & \text{ if }\ \  (R_{\Delta})_{r}=I_{r}\\
-\mathcal{D}_{I}^{v}\left(P_{\Delta}(x)e_{k\ell}(\sigma(x+z))\right), & \text{ if }\ \  (R_{\Delta})_{r}\neq I_{r}
\end{cases}$$
\end{lemma}
\begin{proof}
Denote for convenience $e_{k\ell}$ by $e$. By Lemma \ref{properties of e_rs and epsilon_rs with star product} we have that $g(x):=e(\tau_{r}^\star(\sigma)(x+\tau_{r}(z)))+e(\sigma(x+z))$ is $\tau_{r}$-invariant.  Indeed,
\begin{align*}
g(\tau_{r}(x))&=e(\tau_{r}^\star(\sigma)(\tau_{r}(x)+\tau_{r}(z)))+e(\sigma(\tau_{r}(x)+z))\\
&= e((\tau_{r}^\star(\sigma))\tau_{r}(x+z))+e(\sigma\tau_{r}(x+\tau_{r}(z)))\\
&=e(\sigma\tau_{r}(\tau_{r}(x+z)))+e(\tau_{r}^\star(\sigma)(x+\tau_{r}(z)))\\
&=e(\sigma(x+z))+e(\tau_{r}^\star(\sigma)(x+\tau_{r}(z)))\\
&=g(x)
\end{align*}

We continue the proof considering four cases.
\begin{itemize}
\item[(i)] $(R_{\Delta})_{r}=I_{r}=\{(i_{r},j_{r})\}$.  In particular, $r\in\Sigma_{I}$ and the function  $g_{1}(x)=P_{\Delta}(x)e(\tau_{r}^\star(\sigma)(x+\tau_{r}(z)))-P_{\Delta}(x)e(\sigma(x+z))$ is $\tau_{r}$-invariant. 
\item[(ii)] $(R_{\Delta})_{r}=\emptyset$ and $I_{r}=\{(i_{r},j_{r})\}$.  In particular,  $r\in\Sigma_{I}$ and the function $g_{2}(x)=P_{\Delta}(x)e(\tau_{r}^\star(\sigma)(x+\tau_{r}(z)))+P_{\Delta}(x)e(\sigma(x+z))$ is $\tau_{r}$-invariant. 
\item[(iii)]  $(R_{\Delta})_{r}=\{(i_{r},j_{r})\}$ and $I_{r}=\emptyset$.  In particular, $r\notin\Sigma_{I}$ and the function $g_{3}(x)=P_{\{r\}}(x)P_{\Delta}(x)e(\tau_{r}^\star(\sigma)(x+\tau_{r}(z)))+P_{\{r\}}(x)P_{\Delta}(x)e(\sigma(x+z))$ is $\tau_{r}$-invariant. 
\item[(iv)] $(R_{\Delta})_{r}=\emptyset$ and $I_{r}=\emptyset$.  In particular, $r\notin\Sigma_{I}$ and  the function $g_{4}(x)=P_{\{r\}}(x)P_{\Delta}(x)e(\tau_{r}^\star(\sigma)(x+\tau_{r}(z)))-P_{\{r\}}(x)P_{\Delta}(x)e(\sigma(x+z))$ is $\tau_{r}$-invariant. 
\end{itemize}
In the cases (i) and (ii) we  apply the operator $\mathcal{D}^{v}_{I}$ to the  $\tau_{r}$-invariant functions $g_{1}$ and $g_{2}$. Then, by Lemma \ref{some identities}(iii), we have $\mathcal{D}^{v}_{I}(g_{1})=0$ and $\mathcal{D}^{v}_{I}(g_{2})=0$. For the cases (iii) and (iv) we  apply the operator $\mathcal{D}^{v}_{I\cup R_{\{r\}}}$ to  $g_{3}$ and $g_{4}$.  Then, again by Lemma \ref{some identities}(iii), we have $\mathcal{D}^{v}_{I\cup R_{\{r\}}}(g_{3})=0$ and $\mathcal{D}^{v}_{I\cup R_{\{r\}}}(g_{4})=0$. Finally, since $\mathcal{D}^{v}_{I\cup R_{\{r\}}}(P_{\{r\}}(x)P_{\Delta}(x)e(\tau_{r}^\star(\sigma)(x+\tau_{r}(z))))=\mathcal{D}^{v}_{I}(P_{\Delta}(x)e(\tau_{r}^\star(\sigma)(x+\tau_{r}(z))))$ and $\mathcal{D}^{v}_{I\cup R_{\{r\}}}(P_{\{r\}}(x)P_{\Delta}(x)e(\sigma(x+z)))=\mathcal{D}^{v}_{I}(P_{\Delta}(x)e(\sigma(x+z)))$, we obtain the desired result.
\end{proof}

\begin{corollary}\label{better generalization of 8.7}
Let $I \subseteq R$ and  $\Delta, \Delta' \subseteq \Sigma$ be such that $P_{\Delta}(x)e_{k\ell}(\tau^{\star}_{\Delta'}(\sigma)(x+\tau_{\Delta'}(z))$ and $P_{\Delta}(x)e_{k\ell}(\sigma(x+z))$ are smooth functions.  Then:
$$\mathcal{D}_{I}^{v}\left(P_{\Delta}(x)e_{k\ell}(\tau^{\star}_{\Delta'}(\sigma)(x+\tau_{\Delta'}(z)))\right)=(-1)^{q}\mathcal{D}_{I}^{v}\left(P_{\Delta}(x)e_{k\ell}(\sigma(x+z))\right),
$$
where $q=|\{r\in\Delta'\; | \;  (R_{\Delta})_{r}\neq I_{r}\}|$.
\end{corollary}
\begin{proof}
The identity follows directly from Lemma \ref{Generalization of lemma 8.7}.
\end{proof}

\begin{lemma}\label{relations of the module for any Delta}
Let $I \subseteq R$ and $\Delta \subseteq \Sigma$. Then: 
\begin{equation*}
\mathcal{D}_{I}T(x+\tau_{\Delta}(z))=
(-1)^{p}\mathcal{D}_{I}T(x+z),
\end{equation*}
where $p=|\{r\in\Delta\; | \; I_{r}\neq\emptyset\}|$.
\end{lemma}
\begin{proof}
The identity follows from  the relations (\ref{relations satisfied by vectors on this universal module}).
\end{proof}
\begin{definition}\label{convention for e(sigma) etc} 
For convenience we introduce the following notation for any $I\subseteq R$ and $(\sigma_1,\sigma_2) \in \Phi_{k\ell}\times \Phi_{rs}$, where $k\neq \ell$ and  $r\neq s$:
\begin{eqnarray*}
e_{k\ell } (\sigma_1) &=& e_{k\ell} (\sigma_1 (x+z)),\\
e_{k\ell} (\sigma_1, \sigma_2) &=& e_{k\ell} (\sigma_1 (x + z + \sigma_2 (\varepsilon_{rs}))),\\
\mathcal{D}_{I}T(\sigma_1 + \sigma_2) &=& \mathcal{D}_{I}T(x+z+\sigma_1(\varepsilon_{k\ell}) + \sigma_2(\varepsilon_{rs}))\\
\Omega_{0}&=&\{u\in\Sigma\; | \; z \mbox{ is } \tau_{u}\mbox{-invariant}\}\\
\Omega(\sigma_1)&=&\{u\in\Sigma\; | \; \min\{k,\ell\}\leq k_{u}\leq \max\{k,\ell\}-1\mbox{ and } z+\sigma_1(\varepsilon_{k\ell}) \mbox{ is } \tau_{u}\mbox{-invariant}\}\\
\widetilde{\Omega}(\sigma_1)&=&\{u\in\Sigma\; | \; \min\{k,\ell\}\leq k_{u}\leq \max\{k,\ell\}-1\mbox{ and }\sigma_1\in\Phi_{k\ell}(i_{u})\cup\Phi_{k\ell}(j_{u})\}\\
\Omega_{\sigma_{1}}&=&\Omega(\sigma_{1})\cap\widetilde{\Omega}(\sigma_{1})\\
\widehat{\Omega}(\sigma_{1},\sigma_{2})&=&\Omega_{\sigma_{1}}\cap\widetilde{\Omega}(\sigma_{2})\\
\Phi_{(\sigma_1,\sigma_2)}&=& \{ (\sigma_1',\sigma_2') \in \Phi_{k\ell}\times\Phi_{rs}\; | \; \sigma_1'(\varepsilon_{k\ell})+\sigma_2'(\varepsilon_{rs})=\sigma_1(\varepsilon_{k\ell})+\sigma_2(\varepsilon_{rs})\}.
\end{eqnarray*}
\end{definition}
\begin{remark}\label{rem: intersections}
Note that $\Omega_{\sigma_1}\cap\Omega_{\sigma_{2}}\subseteq \widehat{\Omega}(\sigma_{1},\sigma_{2})$ and $\widehat{\Omega}(\sigma_{1},\sigma_{2})\subseteq \Omega_{\sigma_{1}}$. In particular,  if $\widehat{\Omega}(\sigma_{1},\sigma_{2})=\widehat{\Omega}(\sigma_{2},\sigma_{1})=\Delta$, then $\Omega_{\sigma_{1}}\cap\Omega_{\sigma_{2}}=\Delta$. Also, if $\Omega_{\sigma_{1}}=\Omega_{\sigma_{2}}=\Delta'$, then $\widehat{\Omega}(\sigma_{1},\sigma_{2})=\widehat{\Omega}(\sigma_{2},\sigma_{1})=\Delta'$.
\end{remark}
\begin{lemma}\label{behavior of e_rs with action of tau star}
Let  $k\neq \ell$, $r \neq s$, and $(\sigma_1,\sigma_2)\in\Phi_{k\ell}\times\Phi_{rs}$. Then for each $\Delta\subseteq\Sigma$ and $\bar{\Delta}\subseteq\Omega(\sigma_2)$ we have:
\begin{itemize}
\item[(i)] $e_{k\ell} (\tau^{\star}_{\bar{\Delta}}(\sigma_1),\sigma_2)^{\tau_{\bar{\Delta}}}=e_{k\ell} (\sigma_1,\sigma_2)$.
\item[(ii)] 
$
\mathcal{D}_{I}^{v}\left(P_{\Delta}(x)e_{k\ell}(\tau^{\star}_{\bar{\Delta}}(\sigma_1),\sigma_2)\right)  = (-1)^{|\{u\in\bar{\Delta}\; | \;  (R_{\Delta})_{u}\neq I_{u}\}|}\mathcal{D}_{I}^{v}\left(P_{\Delta}(x)e_{k\ell}(\sigma_1,\sigma_2)\right)
$.
\end{itemize}
\end{lemma}
\begin{proof} 
To prove (i) we use Lemma \ref{properties of e_rs and epsilon_rs with star product}(i) and the fact that $z+\sigma_2 (\varepsilon_{rs})$ is  $\tau_{\bar{\Delta}}$-invariant. Namely, we have:
\begin{align*}
e_{k\ell} (\tau^{\star}_{\bar{\Delta}}(\sigma_1),\sigma_2)^{\tau_{\bar{\Delta}}}&= e_{k\ell} (\tau^{\star}_{\bar{\Delta}}(\sigma_1) (\tau_{\bar{\Delta}}(x) + z + \sigma_2 (\varepsilon_{rs})))\\
&= e_{k\ell} (\sigma_1 \tau_{\bar{\Delta}}(\tau_{\bar{\Delta}}(x) + z + \sigma_2 (\varepsilon_{rs})))\\
&= e_{k\ell} (\sigma_1 (x + \tau_{\bar{\Delta}}(z + \sigma_2 (\varepsilon_{rs}))))\\
&= e_{k\ell} (\sigma_1 (x + z + \sigma_2 (\varepsilon_{rs})))\\
&=e_{k\ell} (\sigma_1,\sigma_2).
\end{align*}
The identity in part (ii) follows  from Corollary \ref{better generalization of 8.7}.
\end{proof}

\begin{lemma}\label{relations in Phi bar} Let $k\neq \ell$,  $r\neq s$  and $(\sigma_1,\sigma_2) \in \Phi_{k\ell}\times \Phi_{rs}$. For any $(\sigma_1',\sigma_2') \in \Phi_{(\sigma_1,\sigma_2)}$ we have:
\begin{itemize}
\item[(i)] $\Phi_{(\sigma_1',\sigma_2')}=\Phi_{(\sigma_1,\sigma_2)}$. 
\item[(ii)] $\Omega(\sigma_1')\cap\Omega(\sigma_2')=\Omega(\sigma_1)\cap\Omega(\sigma_2)$.
\item[(iii)] $\Omega_{\sigma_1'}\cap\Omega_{\sigma_2'}=\Omega_{\sigma_1}\cap\Omega_{\sigma_2}$.
\end{itemize}
\end{lemma}
\begin{proof}
Part (i) follows from the definition of $\Phi_{(\sigma_1,\sigma_2)}$. We next prove  part (ii). Since $(\sigma_{1}',\sigma_{2}')\in \Phi_{(\sigma_1,\sigma_2)}$, for any  $u\in\Omega(\sigma_1)\cap\Omega(\sigma_2)$, we have $\{\sigma_{1}'[k_{u}],\sigma_{2}'[k_{u}]\}=\{\sigma_{1}[k_{u}],\sigma_{2}[k_{u}]\}$. Thus $u\in\Omega(\sigma_1')\cap\Omega(\sigma_2')$, which implies $\Omega(\sigma_1)\cap\Omega(\sigma_2)\subseteq\Omega(\sigma_1')\cap\Omega(\sigma_2')$. For the reverse inclusion, if we start with $u\in\Omega(\sigma_1')\cap\Omega(\sigma_2')$, and use that $(\sigma_{1},\sigma_{2})\in \Phi_{(\sigma_1,\sigma_2)}=\Phi_{(\sigma_1',\sigma_2')}$ (by part (i)) we conclude that $u\in\Omega(\sigma_1)\cap\Omega(\sigma_2)$ using the same reasoning as for the first inclusion. For part (iii), if $(\sigma_{1}',\sigma_{2}')\in \Phi_{(\sigma_1,\sigma_2)}$ and $u\in\Omega_{\sigma_1}\cap\Omega_{\sigma_2}$, then $\sigma_{1}'[k_{u}]=\sigma_{1}[k_{u}]$ and $\sigma_{2}'[k_{u}]=\sigma_{2}[k_{u}]$, and then use again  the reasoning of part (ii).
\end{proof}

\begin{lemma} \label{lm-rs-ident} Let $k\neq \ell$, $r \neq s$, and $(\sigma_1,\sigma_2) \in \Phi_{k\ell}\times\Phi_{rs}$. Let also
$$C(\sigma_1,\sigma_2)=\sum_{\Phi_{(\sigma_1,\sigma_2)}} \big(e_{rs} (\sigma_2') e_{k\ell} (\sigma_1',\sigma_2') - e_{k\ell} (\sigma_1') e_{rs} (\sigma_1',\sigma_2')\big),$$
where the sum is taken over all $(\sigma_{1}',\sigma_2')\in\Phi_{(\sigma_1,\sigma_2)}$. Then the following hold.
\begin{itemize}
\item[(i)] $P_{\Omega_{0}}(x)C(\sigma_1,\sigma_2)$ is a smooth function.
\item[(ii)] If  $\Omega_{\sigma_1}\cap\Omega_{\sigma_2}\neq\emptyset$, then $C(\sigma_1,\sigma_2)=0$
\end{itemize}
\end{lemma}
\begin{proof} Part (i) follows by a straightforward verification. For part (ii) we use the same reasoning as in the case $t=1$ (see Lemma A.4 in \cite{FGR2}). Namely, 
 we use  the fact that $C(\sigma_1,\sigma_2)$ is the coefficient of $T(\sigma_1+\sigma_2)$ in the decomposition of $[E_{k\ell},E_{rs}]T(x)$ as a linear combination of generic tableaux.
\end{proof}

\begin{proposition}\label{existence of tau Deltas}
Set $k\neq \ell$, $r \neq s$ and $(\sigma_{1},\sigma_{2})\in\Phi_{k\ell}\times\Phi_{rs}$. If $\Delta=\widehat{\Omega}(\sigma_1,\sigma_2)$, then there exists $\bar{\Delta}\subseteq\Delta$ such that $\Omega_{\tau^{\star}_{\bar{\Delta}}(\sigma_2)}=\Delta$. \end{proposition}
\begin{proof}
Let  $u\in\Delta$. Since $u\in\Omega_{\sigma_1}$, we  have $|z_{k_{u},i_{u}}-z_{k_{u},j_{u}}|=1$. On the other hand,  since  $u\in\widetilde{\Omega}(\sigma_2)$ we have $u\in\Omega(\sigma_2)$ or $u\in\Omega(\tau_{u}^\star(\sigma_2))$. Thus, $\bar{\Delta}=\{u\in\Delta\; | \; u\in\Omega(\tau_{u}^\star(\sigma_2))\}$ satisfies the desired property.
\end{proof}
\begin{corollary}\label{we can assume singularities are the same}
Set $(\sigma_{1},\sigma_{2})\in\Phi_{k\ell}\times\Phi_{rs}$ and let $\Delta = \widehat{\Omega}(\sigma_{1},\sigma_{2})\cup\widehat{\Omega}(\sigma_{2},\sigma_{1})$. There exist $\Delta_{1}\subseteq\widehat{\Omega}(\sigma_{1},\sigma_{2})\setminus(\Omega_{\sigma_1}\cap\Omega_{\sigma_2})\subseteq\Delta$ and $\Delta_{2}\subseteq\widehat{\Omega}(\sigma_{2},\sigma_{1})\setminus(\Omega_{\sigma_1}\cap\Omega_{\sigma_2})\subseteq\Delta$ such that $\Omega_{\tau^{\star}_{\Delta_{1}}(\sigma_{1})}=\Omega_{\tau^{\star}_{\Delta_{2}}(\sigma_{2})}=\Delta$.
\end{corollary}
\begin{proof}
The statement follows directly from Proposition \ref{existence of tau Deltas}.
\end{proof}


\section{Module structure on $V(T(v))$} \label{sec-mod-str}

Throughout this section we fix $x$ to be an element  in  $\mathcal{S}^{0}$ and $I$ to be a subset of $R$. 
\begin{proposition} \label{d-i-well-defined}
For any $g\in\mathfrak{gl}(n)$ and  $s\in\Sigma$ we have:
$$
\mathcal{D}_{R}^{v}\left(P_{\Sigma\setminus\Sigma_{I}}(x) gT(x+\tau_{s}(z))\right)=
\begin{cases}\mathcal{D}_{R}^{v}\left(P_{\Sigma\setminus\Sigma_{I}}(x) gT(x+ z)\right), & \text{ if } \ \  I_{s}=\emptyset\\
-\mathcal{D}_{R}^{v}\left(P_{\Sigma\setminus\Sigma_{I}}(x) gT(x+ z)\right), & \text{ if }\ \   I_{s}\neq\emptyset.
\end{cases}$$ 
\end{proposition}
\begin{proof} By Remark \ref{evaluation step by step}, for any permutation $\sigma\in S_{t}$ we have:
\begin{multline*}
\mathcal{D}_{R}^{v}\left(P_{\Sigma\setminus\Sigma_{I}}(x) gT(x+\tau_{s}(z))\right)=\\
\mathcal{D}^{v_{\Delta_{t}}}_{R_{\sigma(t)}}\left(\cdots\left(\mathcal{D}^{v_{\Delta_{1}}}_{R_{\sigma(1)}}\left(P_{\Sigma\setminus\Sigma_{I}}(x)gT(x+\tau_{s}(z))\right)\right)\cdots\right)
\end{multline*}
Let $\sigma$ be any permutation such that $\sigma(1)=s$. Since $P_{(\Sigma\setminus\Sigma_{I})\setminus \{s\}}(x)$ does not depend  on $x_{k_{s},i_{s}}$ and $x_{k_{s},j_{s}}$, the proof of the proposition can be completed similarly to the proof of  Proposition 4.7 in \cite{FGR2}.  We have
\begin{multline*}
 \mathcal{D}^{v_{\Delta_{1}}}_{R_{\sigma(1)}}\left(P_{\Sigma\setminus\Sigma_{I}}(x)gT(x+\tau_{s}(z))\right)\\
  =\begin{cases}
 P_{(\Sigma\setminus\Sigma_{I})\setminus \{s\}}(x) \mathcal{D}^{v_{\Delta_{1}}}_{R_{\sigma(1)}}\left(P_{\{s\}}(x)gT(x+\tau_{s}(z))\right), & \text{ if } \ \ I_{s}=\emptyset\\
P_{\Sigma\setminus\Sigma_{I}}(x) \mathcal{D}^{v_{\Delta_{1}}}_{R_{\sigma(1)}}\left(gT(x+\tau_{s}(z)))\right), & \text{ if }\ \ I_{s}\neq\emptyset.
 \end{cases}
\end{multline*}
\end{proof}

Based on  Lemma \ref{derivatives acting on products} and the fact that $\mathcal{D}_{\emptyset}T(v+z) = T(v+z)$, for any $g \in \mathfrak{gl}(n)$ and $I \subseteq R$, we define
\begin{equation} \label{def-action}
g \cdot \mathcal{D}_{I}T(v+z)=\mathcal{D}_{R}^{v}\left(P_{\Sigma\setminus\Sigma_{I}}(x) gT(x+z)\right).
\end{equation}
In order to check that $g \cdot \mathcal{D}_{I}T(v+z)$ is well-defined in $V(T(v))$ we need to verify the independence on the relations (\ref{relations satisfied by vectors on this universal module}) and that the right hand side of (\ref{def-action}) is in  $V(T(v))$.

\begin{lemma}
For $g \in \mathfrak{gl}(n)$ and $z \in  T_{n-1}(\mathbb{Z})$, $g \cdot \mathcal{D}_{I}T(v+z)$ is well-defined.
\end{lemma}
\begin{proof}Note that $P_{\Sigma\setminus\Sigma_{I}}(x) gT(x+z) \in \overline{\mathcal F} \otimes {\mathcal V}_{\rm gen}$, hence the right hand side of (\ref{def-action}) is well-defined. Also, by Proposition \ref{d-i-well-defined}, we verify that  
$g \cdot \mathcal{D}_{I}T(v+\tau_r(z)) = (-1)^{|I_r|}g \cdot \mathcal{D}_{I}T(v+z)$ which implies the independence on (\ref{relations satisfied by vectors on this universal module}).
 \end{proof}

The following theorem shows that $V(T(v))$ has a $\mathfrak{gl}(n)$-module structure. Recall that the  action of the generators $E_{rs}$ on $T(x+z)$ is defined by the formulas (\ref{formula of coefficients e_ij}) in Proposition \ref{coefficients e_ij}.

\begin{theorem}\label{Gelfand-Tsetlin module over gl(n)}
The formulas (\ref{def-action}) endow $V(T(v))$ with a structure of a $\mathfrak{gl}(n)$-module.
\end{theorem}

In order to prove Theorem \ref{Gelfand-Tsetlin module over gl(n)} we will show that for any $I\subseteq R$ and any $1\leq k,\ \ell,\ r,\ s\leq n$ we have the following relations:
\begin{equation}\label{module structure relation}
[E_{k\ell}, E_{rs}] (\mathcal{D}_{I}T(v + z))=E_{k\ell}( E_{rs} (\mathcal{D}_{I}T(v + z))) - E_{rs}(E_{k\ell} (\mathcal{D}_{I}T(v + z)))
\end{equation}

The cases $r=s$  or $k=\ell$ follow by a straightforward computation. Assume now  that $r \neq s$ and $k\neq \ell$. Let $(\sigma_1, \sigma_2) \in \Phi_{k\ell}\times\Phi_{rs}$.  For convenience we will use the following convention (see Definition \ref{convention for e(sigma) etc}):
\begin{eqnarray*}
R_1(\sigma_1, \sigma_2) &=& P_{\Sigma\setminus\Sigma_{I}}(v) e_{rs} (\sigma_2)   e_{k\ell} (\sigma_1, \sigma_2 ) T(\sigma_1 + \sigma_2),\\
R_2(\sigma_1, \sigma_2)  &=&  P_{\Sigma\setminus\Sigma_{I}}(v) e_{k\ell} (\sigma_1)   e_{rs} (\sigma_2, \sigma_1 ) T(\sigma_1 + \sigma_2)\\
L_1(\sigma_1, \sigma_2) &=&
\sum\limits_{\substack{J\subseteq R\\ J\cap R_{\Omega_{\sigma_{2}}}=\emptyset}}{\mathcal D}_{R\setminus J}^{v} (P_{\Sigma\setminus\Sigma_{I}}(v)e_{rs} (\sigma_2)) {\mathcal D}_{R}^{v} \left( P_{\Sigma\setminus\Sigma_{J}}(x) e_{k\ell} (\sigma_1, \sigma_2 ) T(\sigma_1 + \sigma_2)\right)\\
L_2(\sigma_1, \sigma_2)  &=& \sum\limits_{\substack{J\subseteq R\\ J\cap R_{\Omega_{\sigma_{1}}}=\emptyset}}{\mathcal D}_{R\setminus J}^{v} \left(P_{\Sigma\setminus\Sigma_{I}}(x)e_{k\ell} (\sigma_1)) {\mathcal D}_{R}^{v} (P_{\Sigma\setminus\Sigma_{J}}(x) e_{rs} (\sigma_2, \sigma_1 ) T(\sigma_1 + \sigma_2)\right)
\\
\end{eqnarray*}
Note that if $I\cap R_{\Omega_{0}}\neq\emptyset$, then $\mathcal{D}_{I}T(v + z)=0$. Therefore, we will consider $I\subseteq R$ such that $I\cap R_{\Omega_{0}}=\emptyset$. By the definition of the $\mathfrak{gl}(n)$-action on $\mathcal{D}_{I}T(v + z)$ we have:
\begin{equation}\label{eq-left}
E_{k\ell}( E_{rs} (\mathcal{D}_{I}T(v + z))) - E_{rs}(E_{k\ell} (\mathcal{D}_{I}T(v + z))) = \sum\limits_{(\sigma_1, \sigma_2)\in\Phi_{k\ell}\times\Phi_{rs}} (L_1(\sigma_1, \sigma_2) - L_2 (\sigma_1, \sigma_2))
\end{equation}
and 
\begin{equation}\label{eq-right}
[E_{k\ell}, E_{rs}] (\mathcal{D}_{I}T(v + z)) =  \mathcal{D}_{R}^{\overline{v}}(P_{\Sigma\setminus\Sigma_{I}(x)} [E_{k\ell}, E_{rs}] T(x + z))=
\end{equation}
$$
 = \mathcal{D}_{R}^{\overline{v}}  \left( \sum\limits_{(\sigma_1, \sigma_2)\in\Phi_{k\ell}\times\Phi_{rs}} (R_1(\sigma_1, \sigma_2) - R_2 (\sigma_1, \sigma_2))\right).$$

Therefore, to prove equation (\ref{module structure relation}) we need to prove that the right hand sides of (\ref{eq-left}) and (\ref{eq-right}) coincide. This will be a direct consequence of Propositions \ref{Both Omega hat are emptyset} and \ref{some Omega hat is not emptyset} below. In fact, by Corollary \ref{we can assume singularities are the same}, we can assume $\widehat{\Omega}(\sigma_{2},\sigma_{1})=\widehat{\Omega}(\sigma_{1},\sigma_{2})=\Delta$ and consider the two cases $\Delta=\emptyset$ (Proposition \ref{Both Omega hat are emptyset}) and $\Delta\neq\emptyset$  (Proposition \ref{some Omega hat is not emptyset}).

 \begin{proposition}\label{Both Omega hat are emptyset}
Set $k\neq \ell$,  $r\neq s$  and $(\sigma_1,\sigma_{2}) \in \Phi_{k\ell}\times \Phi_{rs}$.
\begin{itemize}
\item[(i)] If $\widehat{\Omega}(\sigma_{2},\sigma_{1})=\emptyset$, then  $L_{1}(\sigma_1,\sigma_2) =  \mathcal{D}_{R}^{{v}} (R_{1}(\sigma_1,\sigma_2))$.
\item[(ii)] If $\widehat{\Omega}(\sigma_{1},\sigma_{2})=\emptyset$, then $L_{2}(\sigma_1,\sigma_2) =  \mathcal{D}_{R}^{{v}} (R_{2}(\sigma_1,\sigma_2))$.
\end{itemize}
\end{proposition}
\begin{proof} Recall  that $\widehat{\Omega}(\sigma,\sigma')=\Omega_{\sigma}\cap\widetilde{\Omega}(\sigma')$ by definition.
\begin{itemize}
\item[(i)]
The hypothesis  $\widehat{\Omega}(\sigma_{2},\sigma_{1})=\emptyset$ implies the following:
\begin{itemize}
\item[(a)] The vector $z+\sigma_{1}(\varepsilon_{k\ell})+\sigma_{2}(\varepsilon_{rs})$ is $\tau_{u}$-invariant for any $u\in\Omega(\sigma_{2})$. In particular, if $K\cap R_{\Omega(\sigma_{2})}\neq\emptyset$, then $\mathcal{D}_{K}T(\sigma_{1}+\sigma_{2})=0$. 
\item[(b)] $ e_{k\ell} (\sigma_1, \sigma_2 )$ is a smooth function.
\item[(c)]By Lemma \ref{behavior of e_rs with action of tau star}(i), $ e_{k\ell} (\sigma_1, \sigma_2 )$ is $\tau_{u}$-invariant for any $u\in \Omega(\sigma_2)$. So, $\mathcal{D}_{J}(e_{k\ell} (\sigma_1, \sigma_2))=0$ whenever $J\cap R_{\Omega(\sigma_{2})}\neq\emptyset$.
\end{itemize}

Now a straightforward computation shows that $\mathcal{D}_{R}^{v}(R_{1}(\sigma_1,\sigma_2))=L_{1}(\sigma_1,\sigma_2)$.
\item[(ii)] The proof is analogous to the proof of part (i).
\end{itemize}
\end{proof}

\begin{proposition}\label{some Omega hat is not emptyset}
Let  $(\sigma_1',\sigma_2')\in\Phi_{k\ell}\times\Phi_{rs}$ be such that $\widehat{\Omega}(\sigma_{2}',\sigma_{1}')=\widehat{\Omega}(\sigma_{1}',\sigma_{2}')=\Delta\neq\emptyset$. Then $L(\sigma'_1,\sigma'_2) =  \mathcal{D}_{R}^{\overline{v}} (R(\sigma'_1,\sigma'_2))$, where $$L(\sigma'_{1},\sigma'_{2})  =
  \sum_{\Phi_{(\sigma_1',\sigma_2')}}\left(\sum_{\substack{\Delta_{1},\Delta_2\subseteq\Delta\\\Delta_{1}\cap\Delta_2=\emptyset}}L_1(\tau^{\star}_{\Delta_1}(\sigma_1), \tau^{\star}_{\Delta_2}(\sigma_2))- L_2(\tau^{\star}_{\Delta_1}(\sigma_1), \tau^{\star}_{\Delta_{2}}(\sigma_2))\right),$$
$$R(\sigma'_1,\sigma'_2) = \sum_{ \Phi_{(\sigma_1',\sigma_2')}}\left(\sum_{\substack{\Delta_{1},\Delta_2\subseteq\Delta\\\Delta_{1}\cap\Delta_2=\emptyset}}R_1(\tau^{\star}_{\Delta_1}(\sigma_1), \tau^{\star}_{\Delta_2}(\sigma_2))- R_2(\tau^{\star}_{\Delta_1}(\sigma_1), \tau^{\star}_{\Delta_{2}}(\sigma_2))\right),$$
and the outer sums on the right hand sides are taken over all $(\sigma_{1},\sigma_{2})\in\Phi_{(\sigma_1',\sigma_2')}$.
\end{proposition}
\begin{proof} In order to show that $L(\sigma'_1,\sigma'_2) =  \mathcal{D}_{R}^{\overline{v}} (R(\sigma'_1,\sigma'_2))$ we will use the followings facts:
\begin{itemize}
\item[(a)] By Remark \ref{rem: intersections} and Lemma \ref{relations in Phi bar}(ii), for each $(\sigma_1,\sigma_2)\in\Phi_{(\sigma_1',\sigma'_2)}$ we have $\Omega_{\sigma_{1}}\cap\Omega_{\sigma_{2}}=\Delta$.
\item[(b)] By (a), for each $\bar{\Delta}\subseteq\Delta$ we have $\Omega_{\tau^{\star}_{\bar{\Delta}}(\sigma_{i})}=\Delta\setminus\bar{\Delta}$, for $i=1,2$.
\item[(c)] By (b) we have $\widehat{\Omega}(\tau^{\star}_{\Delta_1}(\sigma_{1}),\tau^{\star}_{\Delta_2}(\sigma_{2}))=\Delta\setminus\Delta_1$ and $\widehat{\Omega}(\tau^{\star}_{\Delta_2}(\sigma_{2}),\tau^{\star}_{\Delta_1}(\sigma_{1}))=\Delta\setminus\Delta_{2}$. In particular, by Proposition \ref{Both Omega hat are emptyset} we have $L_1(\sigma_1, \tau^{\star}_{\Delta}(\sigma_2))=\mathcal{D}_{R}^{v}(R_1(\sigma_1, \tau^{\star}_{\Delta}(\sigma_2)))$ and $L_2(\tau^{\star}_{\Delta}(\sigma_1), \sigma_2)=\mathcal{D}_{R}^{v}(R_2(\tau^{\star}_{\Delta}(\sigma_1), \sigma_2))$.
\item[(d)] Since $\Omega_{\tau^{\star}_{\Delta_1}(\sigma_{1})}\cap\Omega_{\tau^{\star}_{\Delta_2}(\sigma_{2})}=\Delta\setminus(\Delta_{1}\cup\Delta_2)$, by Lemma \ref{lm-rs-ident}(ii) we have that for any $\Delta_{1},\Delta_{2}$ such that $\Delta_{1}\cup\Delta_{2}\neq\Delta$:
$$\sum_{\Phi_{(\sigma_1',\sigma_2')}}\left(R_1(\tau^{\star}_{\Delta_1}(\sigma_1), \tau^{\star}_{\Delta_2}(\sigma_2))- R_2(\tau^{\star}_{\Delta_1}(\sigma_1), \tau^{\star}_{\Delta_{2}}(\sigma_2))\right)=0.$$
\item[(e)] If $\Delta_{1}\cap\Delta_{2}=\emptyset$,  then for each $(\sigma_{1},\sigma_{2})\in\Phi_{(\sigma'_1,\sigma'_2)}$ we have  $$z+\tau^{\star}_{\Delta_{1}}(\sigma_{1})(\varepsilon_{k\ell})+\tau^{\star}_{\Delta_2}(\sigma_{2})(\varepsilon_{rs})=\tau_{\Delta_{1}\cup\Delta_{2}}(z+\sigma_1(\varepsilon_{k\ell})+\sigma_2(\varepsilon_{rs})).$$
In particular, for any $K\subseteq R$ we can use Lemma \ref{relations of the module for any Delta} to obtain:
\begin{align*}\mathcal{D}_{K}T(\tau^{\star}_{\Delta_{1}}(\sigma_1)+ \tau^{\star}_{\Delta_2}(\sigma_2))&=\mathcal{D}_{K}T(\tau^{\star}_{\Delta_2}(\sigma_1)+ \tau^{\star}_{\Delta_{1}}(\sigma_2))\\
&=(-1)^{|\{r\in\Delta_{1}\cup\Delta_{2}\; | \; K_{r}\neq\emptyset\}|}\mathcal{D}_{K}T(\sigma_1+ \sigma_2).
\end{align*}

\item[(f)] If  $\Delta_{1}\cap\Delta_{2}=\emptyset$, then by Corollary \ref{better generalization of 8.7},  for any $K\subseteq R$ and  $(\sigma_{1},\sigma_{2})\in\Phi_{(\sigma'_1,\sigma'_2)}$ we have: 
\begin{align*}
\mathcal{D}_{R\setminus K}(P_{\bar{\Delta}}(x)e_{k\ell}(\tau^{\star}_{\Delta_{1}}(\sigma_{1}),\tau^{\star}_{\Delta_{2}}(\sigma_{2}))&=(-1)^{s}\mathcal{D}_{R\setminus K}(P_{\bar{\Delta}}(x)e_{k\ell}(\sigma_{1},\tau^{\star}_{\Delta_{2}}(\sigma_{2})),
\end{align*}
where $s=|\{r\in\Delta_1\; | \; (R_{\bar{\Delta}})_{r}\neq (R\setminus K)_{r}\}|$.
\item[(g)] For any $\bar{\Delta}\subseteq\Delta$ the condition $R_{\Omega_{\tau^{\star}_{\bar{\Delta}}(\sigma_{1})}}\cap J=\emptyset$ is equivalent to the condition $J_{r}=\emptyset$ for any $r\in\Delta\setminus\bar{\Delta}$. In particular, if $\Delta_{1}\cap\Delta_{2}=\emptyset$, $R_{\Omega_{\tau^{\star}_{\Delta_{2}}(\sigma_{2})}}\cap J=\emptyset$ implies $J_{r}=\emptyset$ for any $r\in\Delta_{1}$, while $R_{\Omega_{\tau^{\star}_{\Delta_{1}}(\sigma_{1})}}\cap J=\emptyset$ implies $J_{r}=\emptyset$ for any $r\in\Delta_{2}$.
\end{itemize}
We finish the proof of Theorem \ref{Gelfand-Tsetlin module over gl(n)} in four steps.

\noindent {\it Step 1}. We use (c), (d) and after reordering the terms of $L(\sigma'_1,\sigma'_2)$ and $R(\sigma'_1,\sigma'_2)$, we verify that in order to prove 
the identity $L(\sigma'_1,\sigma'_2) =  \mathcal{D}_{R}^{v} (R(\sigma'_1,\sigma'_2))$, it is sufficient  to show $\tilde{L}(\sigma'_1,\sigma'_2) =  \mathcal{D}_{R}^{v} (\tilde{R}(\sigma'_1,\sigma'_2))$, where 
$$\tilde{L}(\sigma'_{1},\sigma'_{2})  =
  \sum_{\Delta_{2}\subsetneq\Delta}\left(\sum_{\substack{\Delta_1\subseteq\Delta\\\Delta_{1}\cap\Delta_{2}=\emptyset}}\sum_{\Phi_{(\sigma_1',\sigma_2')}}L_1(\tau^{\star}_{\Delta_1}(\sigma_1), \tau^{\star}_{\Delta_2}(\sigma_2))- L_2(\tau^{\star}_{\Delta_2}(\sigma_1), \tau^{\star}_{\Delta_{1}}(\sigma_2))\right),$$
$$\tilde{R}(\sigma'_1,\sigma'_2) =  \sum_{\Delta_{2}\subsetneq\Delta}\left(\sum_{ \Phi_{(\sigma_1',\sigma_2')}}R_1(\tau^{\star}_{\Delta\setminus\Delta_2}(\sigma_1), \tau^{\star}_{\Delta_2}(\sigma_2))- R_2(\tau^{\star}_{\Delta_2}(\sigma_1), \tau^{\star}_{\Delta\setminus\Delta_{2}}(\sigma_2))\right).$$

\noindent {\it Step 2}.  We use (e), (f), and (g) to simplify $\tilde{L}(\sigma'_{1},\sigma'_{2})$. Namely,   for any $\Delta_{1}\subseteq\Delta$ and  $\Delta_{2}\subsetneq\Delta$, we have
\begin{multline*}L_1(\tau^{\star}_{\Delta_1}(\sigma_1), \tau^{\star}_{\Delta_2}(\sigma_2))- L_2(\tau^{\star}_{\Delta_2}(\sigma_1), \tau^{\star}_{\Delta_{1}}(\sigma_2))=\\L_1(\tau^{\star}_{\Delta\setminus\Delta_2}(\sigma_1), \tau^{\star}_{\Delta_2}(\sigma_2))- L_2(\tau^{\star}_{\Delta_2}(\sigma_1), \tau^{\star}_{\Delta\setminus\Delta_{2}}(\sigma_2)).
\end{multline*}
Therefore,  $\tilde{L}(\sigma'_{1},\sigma'_{2})$ is equal to
\begin{equation}\label{LHS}
  \sum_{\Delta_{2}\subsetneq\Delta}\left(2^{|\Delta\setminus\Delta_{2}|}\sum_{\Phi_{(\sigma_1',\sigma_2')}}L_1(\tau^{\star}_{\Delta\setminus\Delta_2}(\sigma_1), \tau^{\star}_{\Delta_2}(\sigma_2))- L_2(\tau^{\star}_{\Delta_2}(\sigma_1), \tau^{\star}_{\Delta\setminus\Delta_{2}}(\sigma_2))\right).
  \end{equation}
  
 \noindent  {\it Step 3}. We compute $\mathcal{D}^{v}_{R}(\tilde{R}(\sigma'_1,\sigma'_2))$. In fact, for each $\Delta_{2}\subsetneq\Delta$, by  (d) we obtain:
\begin{multline}\label{RHS}
  \sum_{ \Phi_{(\sigma_1',\sigma_2')}}R_1(\tau^{\star}_{\Delta\setminus\Delta_2}(\sigma_1), \tau^{\star}_{\Delta_2}(\sigma_2))- R_2(\tau^{\star}_{\Delta_2}(\sigma_1), \tau^{\star}_{\Delta\setminus\Delta_{2}}(\sigma_2))=\\P_{\Sigma\setminus\Sigma_{I}}(v)C(\tau^{\star}_{\Delta\setminus\Delta_2}(\sigma'_1), \tau^{\star}_{\Delta_2}(\sigma'_2))T(x+\tau_{\Delta}(z+\sigma'_{1}(\varepsilon_{k\ell})+\sigma'_{2}(\varepsilon_{rs}))),
  \end{multline}
 where $C(\tau^{\star}_{\Delta\setminus\Delta_2}(\sigma'_1), \tau^{\star}_{\Delta_2}(\sigma'_2))$ is equal to
$$
\sum_{\Phi_{(\sigma_1',\sigma_2')}}\big(e_{rs} (\tau^{\star}_{\Delta_2}(\sigma_{2})) e_{k\ell} (\tau^{\star}_{\Delta\setminus\Delta_2}(\sigma_{1}),\tau^{\star}_{\Delta_2}(\sigma_{2})) - e_{k\ell } (\tau^{\star}_{\Delta_2}(\sigma_{1})) e_{rs} (\tau^{\star}_{\Delta_2}(\sigma_2),\tau^{\star}_{\Delta\setminus\Delta_2}(\sigma_{1}))\big).
$$
By Lemma \ref{lm-rs-ident}(i),  $P_{\Sigma\setminus\Sigma_{I}}(v)C(\tau^{\star}_{\Delta\setminus\Delta_2}(\sigma'_1), \tau^{\star}_{\Delta_2}(\sigma'_2))$ is a smooth function (note that $\Omega_{0}\subseteq\Sigma\setminus\Sigma_{I}$ because $I\cap R_{\Omega_{0}}=\emptyset$).

Now, if  $\{(\sigma_{1}^{(p)}, \sigma_{2}^{(p)})\}_{p=1}^{s}$ is the set of all pairs of permutations in $\Phi_{(\sigma_{1}',\sigma_{2}')}$, for each $\Delta_{2}\subsetneq\Delta$ we define the following functions:
\begin{eqnarray*}
 f_{2p} =  P_{\Sigma\setminus\Sigma_{I}}(v)e_{rs}(\tau^{\star}_{\Delta_{2}}(\sigma_{2}^{(p)})), && \ \ 
f_{2p-1} =P_{\Sigma\setminus\Sigma_{I}}(v)e_{k\ell}(\tau^{\star}_{\Delta_{2}}(\sigma_{1}^{(p)})),\\
g_{2p-1} = e_{k\ell}(\tau^{\star}_{\Delta\setminus\Delta_{2}}(\sigma_{1}^{(p)}),\tau^{\star}_{\Delta_{2}}(\sigma_{2}^{(p)})),& &\ \ 
g_{2p} =-e_{rs}(\tau^{\star}_{\Delta_{2}}(\sigma_{2}^{(p)}),\tau^{\star}_{\Delta\setminus\Delta_{2}}(\sigma_{1}^{(p)})).
\end{eqnarray*}

We finally apply Proposition \ref{Most general version of lemma 8.2} to the functions $f_{p},g_{p}$, $p=1,\ldots, 2s$. Note that the hypotheses of Proposition \ref{Most general version of lemma 8.2} are satisfied by Lemma \ref{lm-rs-ident} and Lemma \ref{behavior of e_rs with action of tau star}(i).

\noindent  {\it Step 4}. To complete the proof we show that  (\ref{LHS}) in Step $2$ coincides with the expression obtained by applying $\mathcal{D}^{v}$ to (\ref{RHS}) in Step $3$.
\end{proof}

\section{New irreducible Gelfand-Tsetlin modules of index $2$}\label{sec: Irreducible modules of index 2} In this section we give examples of new irreducible Gelfand-Tsetlin modules of index $2$ 
which are certain irreducible Verma modules.

Take $a_{i}\in\mathbb{C}$, $i=1,\ldots,n-1$  such that $a_{i}-a_{j}\notin\mathbb{Z}$ for any $i\neq j$. Let  $T(v)$  be the Gelfand-Tsetlin tableau with entries $v_{r1}=v_{r2}=a_{1}$ for $1\leq r\leq n$ and $v_{ri}=a_{i-1}$ for $i=3,\ldots,r\leq n$, namely the tableau:
\newpage

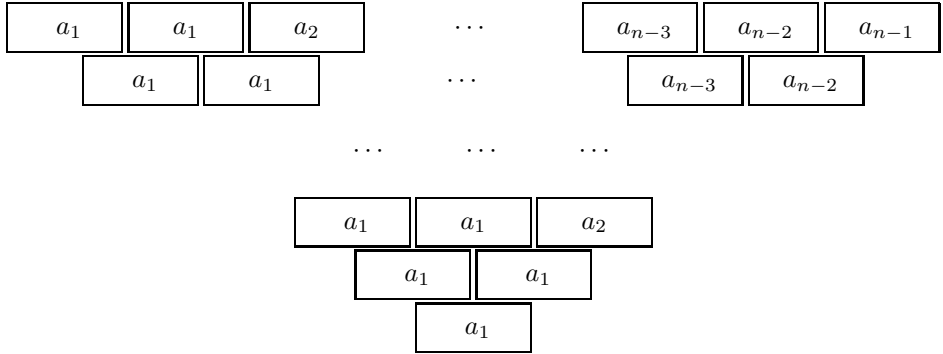
\begin{center}

\Stone{\mbox{ $a_{1}$}}\Stone{\mbox{$a_{1}$}}\Stone{\mbox{$a_2$}}\hspace{1cm} $\cdots$ \hspace{1cm} \Stone{\mbox{ $a_{n-3}$}}\Stone{\mbox{ $a_{n-2}$}}\Stone{\mbox{ $a_{n-1}$}}\\[0.2pt]
\Stone{\mbox{ $a_{1}$}}\Stone{\mbox{$a_{1}$}}\hspace{1.5cm} $\cdots$ \hspace{1.7cm} \Stone{\mbox{ $a_{n-3}$}}\Stone{\mbox{ $a_{n-2}$}}\\[0.5cm]
\hspace{0.2cm}$\cdots$ \hspace{0.8cm} $\cdots$ \hspace{0.8cm} $\cdots$\\[0.5cm]
\Stone{\mbox{ $a_{1}$}}\Stone{\mbox{$a_{1}$}}\Stone{\mbox{$a_2$}}\\[0.2pt]
\Stone{\mbox{ $a_{1}$}}\Stone{\mbox{ $a_{1}$}}\\[0.2pt]
\Stone{\mbox{ $a_{1}$}}\\
\medskip
\end{center}
 Consider the corresponding module $V(T(v))$.  It is an $(n-2)$-singular Gelfand-Tsetlin module of index $2$.      
 
 \begin{theorem}\label{thm-Verma}  Let $T(v)$ be the tableau defined above, and let $\sm:=\sm_{T(v)}$.
 \begin{itemize}
\item[(i)]  The module $V(T(v))$ has a unique irreducible subquotient $M$ such that $M_{\sm}\neq 0$. Moreover, $M$ is a submodule of $V(T(v))$ and it is
 isomorphic to the Verma module with highest weight $(a_{1},a_{1}+1,a_{2}+2,\ldots, a_{n-1}+n-1)$. 
\item[(ii)] $\mbox{GT-}\deg (M)=\mbox{GT-}\deg (V(T(v))=2^{n-2}$.
\item[(iii)] The geometric multiplicities of all eigenvalues of any generator of $\Gamma$ on $M$ are bounded by $2$. The geometric multiplicity of $c_{k2}$ on a Gelfand-Tsetlin subspace of a maximal dimension is exactly $2$, whenever the $k$-th row contains a critical pair.  In particular, the geometric $\mbox{GT}$-degree of $M$ is $2$.
\end{itemize} 
\end{theorem} 

The proof of this theorem will be given in the  Subsection \ref{subsec-Verma}.

\begin{remark}
For $n \geq 4$ the geometric GT-degree of the module $M(a_1,a_1+1,a_{2}+2,\ldots ,a_{n-1}+n-1)$ is strictly smaller than the GT-degree of this module ($2<2^{n-2}$).
\end{remark}

\subsection{Gelfand-Tsetlin degree conjecture}
Let $\rho$ be a half of the sum of positive roots of $\mathfrak{gl}(n)$.  Then the Verma module $M(-\rho)$ is irreducible and it is a singular Gelfand-Tsetlin module of highest index, i.e. it  has index $n-1$.  For $n=3$ it has singularity of index $2$ 
and hence satisfies the theorem above. 

\medskip

\noindent {\bf Conjecture 4.} Consider the irreducible Verma module $M(-\rho)$ of index $n-1$. We conjecture that this module is a Gelfand-Tsetlin module of maximum possible GT-degree, i.e. $\mbox{GT-deg}(M(-\rho)) = (n-1)!(n-2)!...1!$

\section{${\mathcal D}^{v}$-invariance of the $\Gamma$-action on $V(T(v))$}\label{section: action of Gamma}

In this section we study the structure of $V(T(v))$ as a Gelfand-Tsetlin module, and, in particular, the action of the generators of the Gelfand-Tsetlin subalgebra $\Ga$ on $V(T(v))$. The main result of this section is the following

\begin{theorem}\label{action of Gamma on derivative tableaux}
$\Gamma=\Gamma_{n}$ is  $\mathcal{D}^{v}$-invariant.
\end{theorem}

The notion of $\mathcal{D}^{v}$-invariance of $\Gamma$ is intuitively clear, but for the sake of completeness, we define it in more general setting.
\begin{definition}
For each $m\leq n$ we denote by $\Gamma_{m}$ the subalgebra of $\Gamma$ generated by the centers $Z_i$ of $U_{i}$, $1\leq i\leq m$. By $\Sigma(m)$ we denote the set $\{r\in\Sigma\; | \; k_{r}< m\}$.\\
We  say that $\Gamma_{m}$ is \emph{$\mathcal{D}^{v}$-invariant} if for any $c\in\Gamma_m$, any $I\subseteq R$, and any $z\in T_{n-1}(\mathbb{Z})$,
$$c\mathcal{D}_{I}T(v+z)=\mathcal{D}^{v}_{I}\left(cT(x+z)\right).$$
\end{definition}

\begin{lemma}\label{DR becomes DI}
Suppose $I\subseteq R$ and  $z\in T_{n-1}(\mathbb{Z})$. If $f$ is any smooth function then
\begin{equation*}
\mathcal{D}^{v}_{R}\left(P_{\Sigma\setminus\Sigma_{I}}(x)fT(x+z)\right)=\mathcal{D}^{v}_{I}\left(fT(x+z)\right).
\end{equation*}
\end{lemma}
\begin{proof} By definition, $
\mathcal{D}^{v}_{R}\left(P_{\Sigma\setminus\Sigma_{I}}(x)fT(x+z)\right)=\sum\limits_{J\subseteq R}\left(\mathcal{D}^{v}_{R\setminus J}\left(P_{\Sigma\setminus\Sigma_{I}}(x)f\right)\mathcal{D}_{J}T(v+z)\right).
$ Now, using Lemma \ref{derivatives acting on products} and the fact that $R_{\Sigma\setminus\Sigma_{I}}=R\setminus I$, the right hand side of the latter identity becomes $\sum\limits_{J\subseteq I}\left(\mathcal{D}^{v}_{I\setminus J}\left(f\right)\mathcal{D}_{J}T(v+z)\right)=\mathcal{D}^{v}_{I}(fT(x+z))
$.
\end{proof}

Recall that ${\mathcal V}_{\rm gen} := \bigoplus_{v \in \mathcal{S}^{0}} V(T(v))$.  We define a $\mathfrak{gl}(n)$-module structure on $\overline{\mathcal{F}}\otimes {\mathcal V}_{\rm gen}$ by letting $\mathfrak{gl}(n)$ to act trivially on $\overline{\mathcal{F}}$. 

\begin{proposition}\label{D_I commutes with gl(n)}
Let $g$ be any element of $\mathfrak{gl}(n)$ and suppose that $F\in\overline{\mathcal{F}}\otimes \mathcal{V}_{gen}$ is such that $g(F)\in\overline{\mathcal{F}}\otimes \mathcal{V}_{gen}$. Then $g\mathcal{D}^{v}_{I}(F)=\mathcal{D}^{v}_{I}g(F)$
\end{proposition} 
\begin{proof}
Since $\mathcal{D}^{\overline{v}}_{I}$ is linear, it is enough to show the statement for $g = E_{rs}$ and $F = f T(x+z)$ with  generic $x$ and a smooth function $f$.  We have:
\begin{align*}
E_{rs}\mathcal{D}^{v}_{I}(fT(x+z))=&\mathcal{D}^{v}_{R}(P_{\Sigma\setminus\Sigma_{I}}(x)E_{rs}(fT(x+z)))\\
=&\mathcal{D}^{v}_{R}\left(P_{\Sigma\setminus\Sigma_{I}}(x)f(x)\sum_{\sigma\in\Phi_{rs}}e_{rs}(\sigma(x+z))T(x+z+\sigma(\varepsilon_{rs}))\right)\\
=&\sum_{\sigma\in\Phi_{rs}}\mathcal{D}^{v}_{R}\left(P_{\Sigma\setminus\Sigma_{I}}(x)f(x) e_{rs}(\sigma(x+z))T(x+z+\sigma(\varepsilon_{rs}))\right)\\
=&\sum_{\sigma\in\Phi_{rs}}\mathcal{D}^{v}_{I}\left(f(x) e_{rs}(\sigma(x+z))T(x+z+\sigma(\varepsilon_{rs}))\right)\\
=&\mathcal{D}^{v}_{I}E_{rs}(fT(x+z)),
\end{align*}
where the forth equality follows from Lemma  \ref{DR becomes DI}.
\end{proof}


From now to the end of this section we will denote by  $l_{1}<\cdots<l_{\tilde{t}}$ the set of all distinct elements in $\{k_{1},\ldots,k_{t}\}$. We also set $l_{0}:=1$.

\begin{definition}
For each $a\in\mathbb{Z}_{\geq 0}$ and any $l\in\{l_0, l_1,\ldots,l_{\tilde{t}}\}$, we define the following subsets of $ T_{n-1}(\mathbb{Z})$.
\begin{itemize}
\item[(i)] $\mathcal{L}_{a}^{(l)}:=\{z\in T_{n-1}(\mathbb{Z})\; | \; |z_{k_{r},i_{r}}-z_{k_{r},j_{r}}|=a \text{ for any } r \text{ such that } k_{r}\leq l\}$.
\item[(ii)] $\mathcal{L}_{\geq a}^{(l)}:=\bigcup\limits_{k\geq a}\mathcal{L}_{k}^{(l)}$.
\item[(iii)] $\mathcal{L}_{a}:=\mathcal{L}_{a}^{(l_{\tilde{t}})}$ and $\mathcal{L}_{\geq a}:=\mathcal{L}_{\geq a}^{(l_{\tilde{t}})}$
\end{itemize}
\end{definition}

Note that for any $a\in\mathbb{Z}_{\geq 0}$ we have $\mathcal{L}_{a}^{(l_{0})}=\mathcal{L}_{\geq a}^{(l_{0})}= T_{n-1}(\mathbb{Z})$.

\begin{lemma}\label{the formulas are the same out of the boundary}
If $z\in\mathcal{L}^{(l_{p-1})}_{\geq 1}$ and $1\leq r\leq s\leq l_{p}$, then for any $\sigma\in\Phi_{rs}$, the coefficient of $\mathcal{D}_{I}T(v+z+\sigma(\varepsilon_{rs}))$ in the decomposition of $E_{rs}\mathcal{D}_{I}T(v+z)$ is $e_{rs}(\sigma(v+z))$. 
\end{lemma}
\begin{proof}
The statement follows by a direct computation from  the action of $E_{rs}$ on $\mathcal{D}_{I}T(v+z)$ in formulas (\ref{def-action}).
\end{proof}

\begin{proposition}\label{Gamma acts nice almost everywhere}
Suppose $c_{rs}\in \Gamma$  and $z\in T_{n-1}(\mathbb{Z})$. Any of the following two conditions
\begin{itemize}
\item[(i)] $z\in\mathcal{L}_{\geq s}$.
\item[(ii)] $m\leq k_{1}$ (recall that, $2\leq k_{1}\leq \cdots \leq k_{t}$ are fixed).
\end{itemize}
 implies the identity:
\begin{equation}\label{Gamma commutes with D_I}
c_{rs}\mathcal{D}_{I}T(v+z)=\mathcal{D}_{I}^{v}\left(c_{rs}T(x+z)\right).
\end{equation}
\end{proposition}

\begin{proof} 
Note that for any $a\in\mathbb{Z}_{\geq 0}$, if $z\in\mathcal{L}_{\geq a}$ and $\sigma\in\Phi_{rs}$ for some  $1\leq r\leq s\leq n$, then $z+\sigma(\varepsilon_{rs})\in\mathcal{L}_{\geq a-1}$. From this observation one can easily show that each of the conditions  (i) and (ii) implies 
$$\{T(x+z), E_{i_s i_1}T(x+z),\ldots,  E_{i_1 i_2}E_{i_2 i_3}\ldots E_{i_s i_1}T(x+z)\}\subseteq\overline{\mathcal{F}}\oplus\mathcal{V}_{gen},$$
 for any $(i_1,\ldots,i_s)\in \{1,\ldots,r \}^s$. Hence, by Proposition \ref{D_I commutes with gl(n)} we have:
\begin{align*}
c_{rs}\mathcal{D}_{I}T(v+z)&=\sum E_{i_1 i_2}E_{i_2 i_3}\ldots E_{i_s i_1}\mathcal{D}_{I}T(v+z)\\
&=\sum\mathcal{D}^{v}_{I}\left( E_{i_1 i_2}E_{i_2 i_3}\ldots E_{i_s i_1}T(x+z)\right)\\
&=\mathcal{D}^{v}_{I}\left(\sum E_{i_1 i_2}E_{i_2 i_3}\ldots E_{i_s i_1}T(x+z)\right)\\
&=\mathcal{D}^{v}_{I}\left(c_{rs}T(x+z)\right),
\end{align*}
where the sums are taken over all $(i_1,\ldots,i_s)\in \{1,\ldots,r \}^s$.  
\end{proof}

\begin{definition}\label{Def: D-order}

\begin{itemize}
\item[(i)] Let $I,\ J\subseteq R$ and $w_{1},w_{2}\in T_{n-1}(\mathbb{Z})$. We  write $\mathcal{D}_{I}T(v+w_{1})\prec_{\mathcal{D}} \mathcal{D}_{J}T(v+w_{2})$ if $I\subseteq J$ and $w_{1}=\tau_{\Delta}(w_{2})$ for some $\Delta\subseteq\Sigma$. We will refer to $\prec_{\mathcal{D}}$ as the \emph{$\mathcal{D}$-order} on $V(T(v))$. 
\item[(ii)] A maximal element in a finite subset $A$ of derivative tableaux in $V(T(v))$ with respect to the $\mathcal{D}$-order will be called \emph{$\mathcal{D}$-maximal} in $A$. 
\end{itemize}

\end{definition}
\begin{remark}
Note that $\prec_{\mathcal{D}}$ defines a preorder, i.e. $\prec_{\mathcal{D}}$ is reflexive and transitive, but  it  is not antisymmetric. Hence, by a maximal element of a set $A$ of derivative tableaux, we mean an element $b$ in $A$ such that for any $c\in A$ we have $c\prec_{\mathcal{D}}b$.
\end{remark}

\begin{lemma}\label{separation for different characters}
Let $m\leq n$ and assume that $\Gamma_{m-1}$ is  $\mathcal{D}^{v}$-invariant. Let $g\in U_{m}$,  $I\subseteq R$, $w \in  T_{n-1}(\mathbb{Z})$, and let  $g\mathcal{D}_{I}T(v+w)=\sum\limits_{j=0}^{k}a_{j}\mathcal{D}_{I^{(j)}}T(v+w_{j})$, where  $S = \{ \mathcal{D}_{I^{(j)}}T(v+w_{j}) \; | \; j=0,...,k\}$ is a linearly independent set of vectors in $V(T(v))$. Assume also that $\mathcal{D}_{I^{(0)}}T(v+w_{0})$ is $\mathcal{D}$-maximal in $S$. Then there exists $C\in\Gamma_{m-1}$ such that 
\begin{itemize}
\item[(i)] $C\mathcal{D}_{I^{(j)}}T(v+w_{j})=0$, if $w_{j}\neq \tau_{\Delta}(w_{0})$ for any $\Delta\subseteq \Sigma$.
\item[(ii)] $Cg\mathcal{D}_{I}T(v+w)=\mathcal{D}_{I^{(0)}}T(v+w_{0})$.
\end{itemize}
\end{lemma}
\begin{proof}
We first note that since $g\in U_{m}$, we have $(I^{(j)})_{r}=I_{r}$ for any $r$ such that $k_{r}\geq m$. Also,  for $m \leq i \leq n$, the $i$-th   row of  the tableau $T(v+w_{j})$ coincide with the $i$-th row of the tableau $T(v+w)$. So, $w_{j}\neq \tau_{\Delta}(w_{0})$ for any $\Delta\subseteq \Sigma$ implies that the rows of the tableau $T(v+w_{j})$ can not be obtained by a permutation of the entries of the first $m-1$ rows of $T(v+w_{0})$. This implies the existence of $c_{j}\in\Gamma_{m-1}$, $\gamma_{j}\in\mathbb{C}$ and $m_{j}\in\mathbb{Z}_{\geq 0}$ such that $(c_{j}-\gamma_{j})^{m_{j}}\mathcal{D}_{I^{(j)}}T(v+w_{j})=0$ and $(c_{j}-\gamma_{j})^{s}\mathcal{D}_{I^{(0)}}T(v+w_{0})\neq 0$ for any $s\in\mathbb{Z}_{\geq 0}$. We continue with the proof of parts (i) and (ii).
\begin{itemize}
\item[(i)] Set $A:=\{j\in \{1,\ldots,k\}\; | \; w_{j}\neq \tau_{\Delta}(w_{0}) \text{ for any } \Delta\subseteq \Sigma\}$. Then 
$$C:=\prod_{j\in A}(c_{j}-\gamma_{j})^{m_{j}}\in\Gamma_{m-1}$$ 
satisfies the identity $C\mathcal{D}_{I^{(j)}}T(v+w_{j})=0$ for any $j\in A$.
\item[(ii)] It is enough to show that $\mathcal{D}_{I^{(0)}}T(v+w_{0})$ appears with nonzero coefficient in the decomposition of $(c_{j}-\gamma_{j})g\mathcal{D}_{I}T(v+w)$ for any $j\in A$. In fact, since $\Gamma_{m-1}$ is  $\mathcal{D}^{v}$-invariant we have:
\begin{align*}
(c_{j}-\gamma_{j})g\mathcal{D}_{I}T(v+w)&=(c_{j}-\gamma_{j})\sum\limits_{i=0}^{k}a_{i}\mathcal{D}_{I^{(i)}}T(v+w_{i})\\
&=\sum\limits_{i=0}^{k}\left(a_{i}\sum_{J^{(i)}\subseteq I^{(i)}}\mathcal{D}^{v}_{I^{(i)}\setminus J^{(i)}}(c_j(v+w_{i})-\gamma_{j})\mathcal{D}_{J^{(i)}}T(v+w_{i})\right).
\end{align*}
In particular, $\mathcal{D}_{I^{(0)}}T(v+w_{0})$ appears in this decomposition if and only if  $I^{(0)}\subseteq I^{(i)}$ and $w_{i}=\tau_{\Delta}(w_{0})$ for some $0\leq i\leq k$ and some $\Delta\subseteq\Sigma$. This, combined with the $\mathcal{D}$-maximality of $\mathcal{D}_{I^{(0)}}T(v+w_{0})$, implies that $\mathcal{D}_{I^{(0)}}T(v+w_{0})$ appears once in this decomposition and its coefficient is $a_{0}(c_{j}(v+w_{0})-\gamma_{j})\neq 0$.
\end{itemize}  \end{proof}

\begin{definition}\label{definition of the relation}
Given $I,J\subseteq R$, $z,z'\in T_{n-1}(\mathbb{Z})$, and $g\in U$, we will write $\mathcal{D}_{J}T(v+z')\xrightarrow{g}\mathcal{D}_{I}T(v+z)$ if $\mathcal{D}_{J}T(v+z')$ appears with nonzero coefficient in the decomposition of $g\cdot\mathcal{D}_{I}T(v+z)$ as linear combination of tableaux. Also, we will write $\mathcal{D}_{J}T(v+z')\rightarrow\mathcal{D}_{I}T(v+z)$ if $\mathcal{D}_{J}T(v+z')\xrightarrow{g}\mathcal{D}_{I}T(v+z)$ for some $g\in U$.
\end{definition}

\begin{lemma}\label{transitivity of the relation}
Let $m\leq n$ be such that $\Gamma_{m-1}$ is $\mathcal{D}^{v}$-invariant. Let also $z_1,\ z_2,\ z_3\in\mathcal{L}_{\geq 1}$ and $I\subseteq R$. If $g_1,\ g_{2}\in U(\mathfrak{gl}(m))$ are such that $\mathcal{D}_{I}T(v+z_2)\xrightarrow{g_1}\mathcal{D}_{I}T(v+z_1)$ and $\mathcal{D}_{I}T(v+z_{3})\xrightarrow{g_{2}}\mathcal{D}_{I}T(v+z_2)$ then $\mathcal{D}_{I}T(v+z_3)\xrightarrow{g_{3}}\mathcal{D}_{I}T(v+z_1)$ for some $g_{3}\in U_{m}$.
\end{lemma}
\begin{proof}
Note that the action of $U_{m}$ on $\mathcal{D}_{I}T(v+z_1)$ with $z\in\mathcal{L}_{\geq 1}$ produces tableaux of the form $\mathcal{D}_{J}T(v+w)$ with $J\subseteq I$. Since $\mathcal{D}_{I}T(v+z_2)\xrightarrow{g_1}\mathcal{D}_{I}T(v+z_1)$, the coefficient of $\mathcal{D}_{I}T(v+z_2)$ in the decomposition of $g_1\mathcal{D}_{I}T(v+z_1)$ is nonzero. Then by Lemma \ref{separation for different characters} there exists $C_{1}\in\Gamma_{m-1}$ such that $C_{1}g_1\mathcal{D}_{I}T(v+z_1)=\mathcal{D}_{I}T(v+z_2)$. For the same reason,  there exists $C_{2}\in\Gamma_{m-1}$ such that $C_{2}g_{2}\mathcal{D}_{I}T(v+z_{2})=\mathcal{D}_{I}T(v+z_3)$. Therefore $C_{2}g_{2}C_{1}g_1\mathcal{D}_{I}T(v+z_1)=\mathcal{D}_{I}T(v+z_3)$.
\end{proof}

\begin{lemma}\label{generalization of lemma 5.4}
Let $p\in\{1,\ldots,\tilde{t}\}$ be such that $\Gamma_{\ell_{p}}$ is $\mathcal{D}^{v}$-invariant and let $z\in\mathcal{L}^{(l_{p})}_{\geq m}\bigcap \mathcal{L}^{(l_{p-1})}_{\geq m+2}$. There exist $z'\in\mathcal{L}^{(l_{p})}_{\geq m+1}$ and $g\in U_{l_{p}+1}$ such that $\mathcal{D}_{I}T(v+z)\xrightarrow{g}\mathcal{D}_{I}T(v+z')$ for any $I\subseteq R$.
\end{lemma}
\begin{proof}
Suppose $l_{p}=k_{r}=\ldots=k_{r+a}$ and let $w = v+z$. Set also $k=l_{p}$ and $\bar{k}=l_{p-1}$. Assume without loss of generality that $z_{k_{r'},i_{r'}}\geq z_{k_{r'},j_{r'}}$ for any $r'\in\Sigma$ (this can be done because of the relations (\ref{relations satisfied by vectors on this universal module})). For every $0\leq b\leq a$, the condition  $z\in\mathcal{L}^{(\bar{k})}_{\geq m+2}$ implies the existence of $t_{b}$ such that 
$$\begin{cases}
w_{k,i_{r+b}}+1=w_{k-1,s_{k-1}}, & \text{ for some \ \ \ $1\leq s_{k-1}\leq k-1$},\\
w_{k-1,s_{k-1}}+1=w_{k-2,s_{k-2}}, & \text{ for some \ \ \ $1\leq s_{k-2}\leq k-2$},\\
\ \ \ \ \ \ \ \ \ \ \ \ \ \ \  \vdots & \ \ \ \ \ \  \vdots\\
w_{k-t_{b}+1,s_{k-t_{b}+1}}+1=w_{k-t_{b},s_{k-t_{b}}}, & \text{ for some \ \ \ $1\leq s_{k-t_{b}}\leq k-t_{b}$},\\
w_{k-t_{b},s_{k-t_{b}}}+1\neq w_{k-t_{b}-1,s}, & \text{ for any \ \ \ $1\leq s\leq k-t_{b}-1$}.\\
\end{cases}
$$

For any $0\leq j\leq a$ set $g_{j}=E_{k+1,k-t_{j}}$ and $z_{j}=z+\delta^{k,i_{r_{1}+j}}+\sum_{r=1}^{t_{j}}\delta^{k-j, s_{k-r}}$. If $b\neq a$ we have $\sum_{j=1}^{b}z_{j}\in\mathcal{L}^{(k)}_{\geq m}\bigcap \mathcal{L}^{(\bar{k})}_{\geq m+1}$ and $z'=\sum_{j=1}^{a}z_{j}\in\mathcal{L}^{(k)}_{\geq m+1}\bigcap \mathcal{L}^{(\bar{k})}_{\geq m+1}=\mathcal{L}^{(k)}_{\geq m+1}$. Now, by Lemma \ref{the formulas are the same out of the boundary} and the choice of $t_{b}$, we have:
$$\mathcal{D}_{I}T(v+z)\xrightarrow{g_{0}}\mathcal{D}_{I}T(v+z_{0})\xrightarrow{g_{1}}\mathcal{D}_{I}T(v+z_{0}+z_{1})\xrightarrow{g_{2}}\cdots\xrightarrow{g_{a}}\mathcal{D}_{I}T(v+z_{0}+z_{1}+\ldots+z_{a}).$$
Finally, by Lemma \ref{transitivity of the relation} we have  $\mathcal{D}_{I}T(v+z)\rightarrow\mathcal{D}_{I}T(v+z_{0}+z_{1}+\ldots+z_{a}).$
\end{proof}

\begin{corollary}\label{generalization of lemma 5.4 with M}

Let $p\in\{1,\ldots,\tilde{t}\}$ be such that $\Gamma_{\ell_{p}}$ is $\mathcal{D}^{v}$-invariant and let $z\in\mathcal{L}^{(l_{p})}_{\geq m}\bigcap \mathcal{L}^{(l_{p-1})}_{\geq m+2M}$ for some  $M\in\mathbb{Z}_{> 0}$. Then there exist $z'\in\mathcal{L}^{(l_{p})}_{\geq m+M}$ and $g\in U_{l_{p}+1}$ such that $\mathcal{D}_{I}T(v+z)\xrightarrow{g}\mathcal{D}_{I}T(v+z')$ for any $I\subseteq R$.
\end{corollary}
\begin{proof}
The statement follows directly from Lemma \ref{generalization of lemma 5.4}.
\end{proof}

\begin{proposition}\label{We can go from L_m to L_m-1}
Assume that $\Gamma_{l_{p}}$ is  $\mathcal{D}^{v}$-invariant and let $z\in\mathcal{L}^{(l_{p})}_{\geq m}$ for some $m\in\mathbb{Z}_{\geq 0}$. Then there exist $z'\in\mathcal{L}^{(l_{p})}_{\geq m+1}$ and $g\in U_{l_{p}+1}$ such that $\mathcal{D}_{I}T(v+z)\xrightarrow{g}\mathcal{D}_{I}T(v+z')$ for any $I\subseteq R$.
\end{proposition}
\begin{proof}
We will prove the existence of $g_{i}\in U_{l_{i}+1}$ and  $z_{i}\in\mathcal{L}^{(l_{i})}_{\geq m+2^{p-i}}$, $i=1,\ldots,p$,  such that:
\begin{itemize} 
\item[(i)] $z_{i}\in\mathcal{L}^{(l_{i+1})}_{\geq m}$, for any $1\leq i\leq p-1$;
\item[(ii)] $\mathcal{D}_{I}T(v+z)\xrightarrow{g_{1}}\mathcal{D}_{I}T(v+z_{1})\xrightarrow{g_{2}}\cdots\xrightarrow{g_{j}}\mathcal{D}_{I}T(v+z_{j}).$
\end{itemize}

We first note that $z\in\mathcal{L}^{(l_{1})}_{\geq m}=\mathcal{L}^{(l_{1})}_{\geq m}\bigcap\mathcal{L}^{(l_0)}_{\geq m+N}$ for any $N\in\mathbb{Z}_{\geq 0}$. Set $N=2^{p}$. By Corollary \ref{generalization of lemma 5.4 with M}, there exist $z_{1}\in\mathcal{L}^{(l_{1})}_{\geq m+2^{p-1}}$ and  $g_{1}\in U_{l_{1}+1}$ such that $\mathcal{D}_{I}T(v+z)\xrightarrow{g_{1}}\mathcal{D}_{I}T(v+z_1)$. Now, assume that for $j$, $1\leq j\leq p-1$, there exist $g_{i}\in U_{l_{i}+1}$, $z_{i}\in\mathcal{L}^{(l_{i})}_{\geq m+2^{p-i}}$, for all $i=1,\ldots,j$, that satisfy (i) and (ii). In particular, $z_{j}\in\mathcal{L}^{(l_{j})}_{\geq m+2^{p-j}}\bigcap\mathcal{L}^{(l_{j+1})}_{\geq m}$ so, we can use Corollary \ref{generalization of lemma 5.4 with M} and guarantee the existence of $g_{j+1}\in U_{l_{j+1}+1}$ and $z_{j+1}\in\mathcal{L}^{(l_{j+1})}_{\geq m+2^{s-j-1}}$ such that $\mathcal{D}_{I}T(v+z_{j})\xrightarrow{g_{j+1}}\mathcal{D}_{I}T(v+z_{j+1})$. Note that $z_{j+1}\in\mathcal{L}^{(l_{j+2})}_{\geq m}$ because $z\in\mathcal{L}^{(l_{p})}_{\geq m}$, $g_{j+1}\in U_{l_{j+1}+1}$, and $\mathcal{D}_{I}T(v+z_{j})\xrightarrow{g_{j+1}}\mathcal{D}_{I}T(v+z_{j+1})$. Finally, the existence of $g\in U_{l_{p}+1}$ such that $\mathcal{D}_{I}T(v+z)\xrightarrow{g}\mathcal{D}_{I}T(v+z')$ is guaranteed by Lemma \ref{transitivity of the relation}.
\end{proof}

\begin{definition}\label{Def: separation of tableaux}
Given $A_{1},\ A_{2}\in V(T(v))$ and $g\in U$,  we say that $g$ \emph{separates} $A_{1}$ and $A_{2}$ if $g A_{1}=A_{1}$ and $g A_{2}=0$.
\end{definition}

\begin{lemma}\label{action of Gamma_k induces action of Gamma_k+1}
Let $m \leq n$ and let $k=\max\{k_{r}\; | \; k_{r}<m\}$. If  $\Gamma_{k}$ is  $\mathcal{D}^{v}$-invariant, then $\Gamma_{m}$ is  $\mathcal{D}^{v}$-invariant.
\end{lemma}

\begin{proof}

By Proposition \ref{Gamma acts nice almost everywhere}(i), if $z\in\mathcal{L}_{\geq m}$, then the formula (\ref{Gamma commutes with D_I}) holds for any $c\in\Gamma_{m}$. Thus $\mathcal{D}_{I}T(v+z)$ is a common eigenvector of all generators of $\Gamma_{m}$ and the submodule $W^{(I)}_{z}$ of $V(T(v))$ generated by $\mathcal{D}_{I}T(v+z)$ is a Gelfand-Tsetlin $\mathfrak{gl}(m)$-module by  Lemma 3.4 in \cite{FGR2}. Then for each $I\subseteq R$, $W_{I}=\sum\limits_{ z\in \mathcal{L}_{\geq m}}W^{(I)}_{z}$ and  $W=\sum\limits_{I\subseteq R}W_{I}$  are also Gelfand-Tsetlin $\mathfrak{gl}(m)$-modules.  Denote by $W_{k}$ the Gelfand Tsetlin $\mathfrak{gl}(m)$-module $\sum\limits_{|\Sigma_{I}\cap\Sigma(m)|=k}W_{I}$. 

We next show that $W_{0}$ contains all tableaux $\mathcal{D}_{I}T(v+z)$ such that $z\in T_{n-1}(\mathbb{Z})$ and $\Sigma_{I}\cap\Sigma(m)= \emptyset$. Let us first consider $z\in\mathcal{L}^{(k)}_{\geq m-1}$. By Proposition \ref{We can go from L_m to L_m-1}, there exist $z_{0}\in\mathcal{L}^{(k)}_{\geq m}$ and $g_{0}\in U_{k+1}$ such that $\mathcal{D}_{I}T(v+z)\xrightarrow{g_{0}} \mathcal{D}_{I}T(v+z_{0})$ (we  assume without loss of generality that the coefficient of $\mathcal{D}_{I}T(v+z)$ in the decomposition of $g_{0} \mathcal{D}_{I}T(v+z_{0})$ is $1$). By Proposition \ref{Gamma acts nice almost everywhere}(i), all generators of $\Gamma_{m}$, except for the ones in the center of $U_{m}$ satisfy the relation (\ref{Gamma commutes with D_I}). Let $c\in Z_{m}$ and let $(c-\gamma)\mathcal{D}_{I}T(v+z_{0})=0$ for some  $\gamma \in {\mathbb C}$. Then $(c-\gamma)g_{0}\mathcal{D}_{I}T(v+z_{0})=0$.  Since $\Gamma_{k}$ is  $\mathcal{D}^{v}$-invariant, we can use Lemma \ref{separation for different characters}, and choose $C\in \Gamma_{k}$ that separates  $\mathcal{D}_{I}T(v+z)$ and $g_{0}\mathcal{D}_{I}T(v+z_{0})-\mathcal{D}_{I}T(v+z)$ (see Definition \ref{Def: separation of tableaux}). Since $C$ commutes with $(c-\gamma)$ we have $(c-\gamma)(Cg_{0}\mathcal{D}_{I}T(v+z_{0}))=0$, which implies that $c$ acts as multiplication by $\gamma$ on any tableau in the decomposition of $g_{0}\mathcal{D}_{I}T(v+z_{0})$. Hence, the action of $\Gamma$ on any $\mathcal{D}_{I}T(v+z)$ for $z\in \mathcal{L}^{(k)}_{\geq m-1}$ is given by (\ref{Gamma commutes with D_I}). Moreover, $\mathcal{D}_{I}T(v+z)\in W$ for any $z\in \mathcal{L}^{(k)}_{\geq m-1}$. 
  Next we consider a tableau $\mathcal{D}_{I}T(v+z)$ with $z\in \mathcal{L}^{(k)}_{\geq m-2}$. Again by Proposition~\ref{We can go from L_m to L_m-1} one finds a nonzero $g_1\in U_{k+1}$ and $z_{1}\in \mathcal{L}^{(k)}_{\geq m-1}$ such that  $\mathcal{D}_{I}T(v+z)\xrightarrow{g_{1}}\mathcal{D}_{I}T(v+z_1)$. For the generators of the center of $U_i$, $i\leq m-2$ the statement follows from Proposition \ref{Gamma acts nice almost everywhere}(i). If $c$ is in the center of $U_{m}$ or in the center of $U_{m-1}$ then it commutes with $g_1$.  Choose $C\in \Gamma_{k}$ that separates  $\mathcal{D}_{I}T(v+z)$ and $g_{0}\mathcal{D}_{I}T(v+z_{0})-\mathcal{D}_{I}T(v+z)$ and which acts by a scalar on the tableau  $\mathcal{D}_{I}T(v+z_1)$. By applying the argument above we conclude that the action of $\Gamma$ on any $\mathcal{D}_{I}T(v+z)$ with $z\in \mathcal{L}^{(k)}_{\geq m-2}$ is determined  by (\ref{Gamma commutes with D_I}) and $\mathcal{D}_{I}T(v+z)\in W$ for any $z\in \mathcal{L}^{(k)}_{\geq m-2}$. Continuing analogously with the sets  $\mathcal{L}^{(k)}_{\geq m-3}, \ldots, \mathcal{L}^{(k)}_{\geq 0}$ we show that any tableau $\mathcal{D}_{I}T(v+z)$ with $\Sigma_{I}\cap\Sigma(m)=\emptyset$ belongs to $W$.
   
 Now consider the  quotient $W_{\geq 1}=W/W_{0}$ and $I$ such that $|\Sigma_{I}\cap\Sigma(m)|=1$. The vector $\mathcal{D}_{I}T(v+z)+W_0$ of $W_{\geq 1}$ is a common eigenvector of $\Gamma_{m}$ by Proposition \ref{Gamma acts nice almost everywhere}(i) for any $z\in \mathcal{L}^{(k)}_{\geq m}$. We can repeat the argument above and obtain that any tableau $\mathcal{D}_{I}T(v+z)$ with $|\Sigma_{I}\cap\Sigma(m)|=1$ belongs to $W$. Continuing in the same fashion when $|\Sigma_{I}\cap\Sigma(m)| = 2,3,\ldots $, we obtain that $\Gamma_{m}$ is  $\mathcal{D}^{v}$-invariant.   \end{proof}

Now we are in the position to prove Theorem \ref{action of Gamma on derivative tableaux}.\\

\noindent{\bf Proof of Theorem \ref{action of Gamma on derivative tableaux}}
Recall that $l_{1}<\cdots<l_{\tilde{t}}$ are the distinct elements of $\{k_{1},\ldots,k_{t}\}$. By Proposition \ref{Gamma acts nice almost everywhere}(ii), $\Gamma_{l_{1}}$ is  $\mathcal{D}^{v}$-invariant. Now, we apply $\bar{t}$ times Lemma \ref{action of Gamma_k induces action of Gamma_k+1} to complete the proof (here $(k,m)\in\{(l_{1},l_{2}),(l_{2},l_{3}),\ldots,(l_{\tilde{t}},n)\}$). \\

\begin{corollary}\label{separation of tableaux if all singular pairs are in different rows}
Suppose that all singular pairs of $T(v)$ are in different rows (i.e. $k_{1}<k_{2}<\ldots < k_{t}$). If $\mathcal{D}_{J}T(v+z')\xrightarrow{g} \mathcal{D}_{I}T(v+z)$,  then there exists $C\in\Gamma$ such that $Cg \mathcal{D}_{I}T(v+z)=\mathcal{D}_{J}T(v+z')$.
\end{corollary}
\begin{proof}
Since $\Gamma$ is $\mathcal{D}^{v}$-invariant, tableaux with different Gelfand-Tsetlin characters can be separated by elements of $\Gamma$. Finally, since all singularities are in different rows, any linearly independent set of vectors in $V(T(v))$ has a $\mathcal{D}$-maximal element.  In particular any linearly independent set of tableaux in the same Gelfand-Tsetlin subspace has a $\mathcal{D}$-maximal element, so using Lemma \ref{separation for different characters} we finish the proof. 
\end{proof}

\begin{proposition}\label{size of jordan blocks}
The following hold.
\begin{itemize}
\item[(i)] The action of $\Gamma$ on $V(T(v))$ is given by the following formulas.
\begin{equation}\label{c_mk commutes with D_I}
c_{ij}\mathcal{D}_{I}T(v+z)=\sum_{J\subseteq I}\mathcal{D}^{v}_{J}\left(\gamma_{ij}(v+z)\right)\mathcal{D}_{I\setminus J}T(v+z)
\end{equation}
\item[(ii)] We have $(c_{ij}-\gamma_{ij}(v+z))^{|I|+1}\mathcal{D}_{I}T(v+z)=0$. In particular, $V(T(v))$ is a Gelfand-Tsetlin module.
\end{itemize}
\end{proposition}
\begin{proof} The identity (\ref{c_mk commutes with D_I}) follows by Theorem \ref{action of Gamma on derivative tableaux}.\\
To prove (ii), we apply induction on $|I|$. Suppose first that $|I|=0$. Then from (\ref{c_mk commutes with D_I}) we obtain $c_{ij}\mathcal{D}_{\emptyset}T(v+z)=\gamma_{ij}(v+z)\mathcal{D}_{\emptyset}T(v+z).$
Suppose now that $|I|=s$ and that $(c_{ij}-\gamma_{ij}(v+z))^{|J|+1}\mathcal{D}_{J}T(v+z)=0$ for any $|J|\leq s-1$. By (\ref{c_mk commutes with D_I}) we have
$$(c_{ij}-\gamma_{ij}(v+z))\mathcal{D}_{I}T(v+z)=\sum_{\emptyset\neq J\subseteq I}\mathcal{D}^{v}_{J}\left(\gamma_{ij}(v+z)\right)\mathcal{D}_{I\setminus J}T(v+z).$$
Since all  subsets $I\setminus J$ with $J\neq \emptyset$ satisfy $|I\setminus J|\leq s-1$, by the induction hypothesis $(c_{ij}-\gamma_{ij}(v+z))^{s}\mathcal{D}_{I\setminus J}T(v+z)=0$. Therefore,
$(c_{ij}-\gamma_{ij}(v+z))^{s+1}\mathcal{D}_{I}T(v+z)=0.$
\end{proof}

\subsection{Proof of Theorem \ref{thm-Verma}}\label{subsec-Verma}
We use notations from Section \ref{sec: Irreducible modules of index 2}. 
Consider the module $M$ generated by the tableau $\mathcal{D}_{\emptyset}T(v)$. It is a highest weight module of highest weight $(a_{1},a_{1}+1,a_{2}+2,\ldots, a_{n-1}+n-1)$. Indeed, for any $1\leq i\leq n-1$, $E_{i,i+1}(\mathcal{D}_{\emptyset}T(v))$ is a linear combination of derivative tableaux $\mathcal{D}_{R\setminus J}$ with coefficients $\mathcal{D}^{v}_{R}(P_{\Sigma\setminus\Sigma_{J}}(v)e_{i,i+1}(\sigma(v)))$, and that coefficient is zero for any $J\subseteq R$. Clearly $\mathcal{D}_{\emptyset}T(v)$ is a weight vector with weight $(a_{1},a_{1}+1,a_{2}+2,\ldots, a_{n-1}+n-1)$. Hence, $M$ is isomorphic to the corresponding irreducible Verma module.  
Since all singularities of $T(v)$ are in different rows, we can apply Corollary \ref{separation of tableaux if all singular pairs are in different rows} and obtain a basis of $M$ given by
 $$\{\mathcal{D}_{I}T(v+z)\in V(T(v))\ |\ \mathcal{D}_{I}T(v+z)\rightarrow\mathcal{D}_{\emptyset}T(v)\}.$$
 Checking the coefficients of the formulas in Theorem \ref{Gelfand-Tsetlin module over gl(n)}, we immediately see that 
 $$\mathcal{D}_{\emptyset}T(v)\xrightarrow{E_{21}}\mathcal{D}_{\emptyset}T(v+\delta^{11})\xrightarrow{E_{32}}\mathcal{D}_{R_{\{2\}}}T(v+\delta^{11}+\delta^{21}).$$
 Also, for any $j=2,\ldots,n-1$, we have $$\mathcal{D}_{R_{\{2,\ldots,j\}}}T\left(v+\sum_{i=1}^{j}\delta^{i1}\right)\xrightarrow{E_{j+2,j+1}}\mathcal{D}_{R_{\{2,\ldots,j+1\}}}T\left(v+\sum_{i=1}^{j+1}\delta^{i1}\right).$$ This, together with Corollary \ref{separation of tableaux if all singular pairs are in different rows}, implies that the tableau $\mathcal{D}_{R}T\left(v+\sum_{i=1}^{n-1}\delta^{i1}\right)$ is a basis element of the module $M$. Finally, if for any $I\subseteq R$, 
$$C_{I}:=\prod_{i\in \Sigma\setminus \Sigma_I}\left(c_{i2}-\gamma_{i2}\left(v+\sum_{i=1}^{n-1}\delta^{i1}\right)\right),$$ then $C_{I}\mathcal{D}_{R}T\left(v+\sum_{i=1}^{n-1}\delta^{i1}\right)$ is a nonzero multiple of $\mathcal{D}_{I}T\left(v+\sum_{i=1}^{n-1}\delta^{i1}\right)$. Hence for any $I\subseteq R$, $\mathcal{D}_{I}T\left(v+\sum_{i=1}^{n-1}\delta^{i1}\right)$ is a basis element of $M$. Now, if $\sn\in \Specm \Gamma$ corresponds to the tableau $\mathcal{D}_{R}T\left(v+\sum_{i=1}^{n-1}\delta^{i1}\right)$, then we have $\dim M_{\sn}=2^{n-2}$. The remaining statements  follow directly from the properties of $V(T(v))$.

\section{Proofs of main theorems} \label{sec-proofs}

\

\noindent{\bf Proof of Theorem A.}\label{subsection: Proof of Th A}
Note that $V(T(v))$ is a Gelfand-Tsetlin module by Theorem \ref{Gelfand-Tsetlin module over gl(n)} and Proposition \ref{size of jordan blocks}. Also, the dimension of $V(T(v))_{\sm}$ coincides with the number of tableaux in $V(T(v))$ having the same Gelfand-Tsetlin character as $\sm$. This completes the proof.\\

\bigskip

\noindent{\bf Proof of Theorem B.}\label{subsection: Proof of Th B}    
For part (i), let us consider a tableau $\mathcal{D}_{I}T(v+z)$ associated with $\sm_{L'}$. A straightforward computation shows that $(c_{kj}-\gamma_{kj}(v+z))^{s}(\mathcal{D}_{I}T(v+z))$ equals the following sum $$\sum_{\emptyset\neq J_{s}\subsetneq J_{s-1}\subsetneq\cdots\subsetneq J_{1}\subsetneq I}\mathcal{D}^{v}_{I\setminus J_1}(\gamma_{kj}(x+z))\cdots \mathcal{D}^{v}_{J_{s-1}\setminus J_s}(\gamma_{kj}(x+z))\mathcal{D}^{v}_{J_s}(\gamma_{kj}(x+z))T(v+z).$$ 

Let $(i_{k_{1}}, j_{k_{1}}),\ldots,(i_{k_{s}}, j_{k_{s}})$ be the singular pairs of $v+z$ on row $k$. If $K\subseteq R$ is such that $K_{r}\neq \emptyset$ for some $r\in\Sigma\setminus\{k_{1},\ldots,k_{s}\}$, then $\mathcal{D}^{v}_{K}(\gamma_{kj}(x+z))=0$ (note that $\gamma_{kj}(v+z)$ depends only of the entries of row $k$). Hence,  there is a  nonzero constant $C$, such that $$(c_{kj}-\gamma_{kj}(v+z))^{s}(\mathcal{D}_{I}T(v+z))=C\mathcal{D}^{v}_{I_{k_{1}}}(\gamma_{kj}(x+z))\cdots\mathcal{D}^{v}_{I_{k_{s}}}(\gamma_{kj}(x+z))T(v+z).$$
From the previous equality we obtain $(c_{kj}-\gamma_{kj}(v+z))^{s+1}(\mathcal{D}_{I}T(v+z))=0$. Also, since $\mathcal{D}^{v}_{I_{k_{i}}}(\gamma_{k2}(x+z))\neq 0$ for any $i=1,\ldots, s$ (see \cite{FGR2}, Lemma 5.2(ii)) we have $(c_{k2}-\gamma_{k2}(v+z))^{s}(\mathcal{D}_{I}T(v+z))\neq 0$.\\

We now prove part (ii). Given $t$-singular vectors $v, v'\in \mathbb C^{\frac{n(n+1}{2}}$ there exist   $z,\ z'\in T_{n-1}(\mathbb{Z})$ such that $\mathcal{D}_{\emptyset}T(v+z)$ and $\mathcal{D}_{\emptyset}T(v'+z')$ are in the fiber of the  maximal ideals $\sm_{v}$ and $\sm_{v'}$ respectively. We have $T(v)-\sigma(T(v'))\in T_{n-1}(\mathbb Z)$. Hence, $V(T(v))\simeq V(T(v'))$. Conversely, let $V(T(v))\simeq V(T(v'))$ for some $t$-singular vectors $v$ and $v'$. Let $\phi$ be any isomorphism between $V(T(v))$ and  $V(T(v'))$. The image of $\mathcal{D}_{\emptyset}T(v)$ under $\phi$ need to satisfies $(c_{rs}-\gamma_{rs}(v))\phi(\mathcal{D}_{\emptyset}T(v))=0$ for any $1\leq s\leq r\leq n$. This implies that $\phi(\mathcal{D}_{\emptyset}T(v))=a\mathcal{D}_{\emptyset}T(\sigma'(v))$ for some $\sigma'\in S_{n-1}\times\cdots\times S_{1}$ and $a\in\mathbb{C}$. Therefore $V(T(v'))\simeq V(T(\sigma'(v)))$ implying $v'-\sigma'(v)\in \mathbb C^{\frac{n(n+1}{2}}$. Now, the image of   $T(v')-\sigma'(T(v))$ via the identification between $T_{n-1}(\mathbb{Z})$ and $ \mathbb{Z}^{\frac{n(n-1)}{2}}$ is $v'-\sigma'(v)$. This completes the proof.\\

Next we prove (iii).
Let $T(v)$ be a Gelfand-Tsetlin tableau. By Corollary \ref{separation of tableaux if all singular pairs are in different rows}, in order to prove the irreducibility of $V(T(v))$ it is sufficient to prove that given any two tableaux $\mathcal{D}_{J}T(v+z) , \mathcal{D}_{J'}T(v+w)$ in $V(T(v))$, we have $\mathcal{D}_{J'}T(v+z)\rightarrow\mathcal{D}_{J}T(v+w)$. 

We have the following two important observations.
\begin{itemize}
\item[(i)] From the proof of Corollary \ref{separation of tableaux if all singular pairs are in different rows}, we have $\mathcal{D}_{J}T(v+w)\rightarrow\mathcal{D}_{I}T(v+w)$ for any $J\subseteq I$ and $w\in T_{n-1}(\mathbb{Z})$. 
\item[(ii)] Since $L$ is regular, we have $\mathcal{D}_{J}T(v+w')\rightarrow\mathcal{D}_{J}T(v+w)$ for any $J\subseteq R$ and $w,\ w'\in T_{n-1}(\mathbb{Z})$.
\end{itemize}
 From (i) and (ii) we conclude that $\mathcal{D}_{I}T(v+w'')\rightarrow\mathcal{D}_{R}T(v+w')$ for any $I\subseteq R$ and $w',\ w''\in T_{n-1}(\mathbb{Z})$. Therefore, to finish the proof we need to prove that for any $\mathcal{D}_{J}T(v+w)\in V(T(v))$ there exists $w''\in T_{n-1}(\mathbb{Z})$ such that $\mathcal{D}_{R}T(v+w'')\rightarrow\mathcal{D}_{J}T(v+w)$.

Consider $w'\in T_{n-1}(\mathbb{Z})$ such that $w'_{k_{r},i_{r}}=w'_{k_{r},j_{r}}$ for any $r\in\Sigma\setminus \Sigma_{J}$. By (ii) we have $\mathcal{D}_{J}T(v+w')\rightarrow\mathcal{D}_{J}T(v+w)$. On the other hand, $$E_{n1}(\mathcal{D}_{J}T(v+w'))=\mathcal{D}_{R}(P_{\Sigma\setminus\Sigma_{J}}(v)E_{n1}(T(v+w'))).$$
Since all singularities are in different rows, there exist $\sigma\in S_{n}\times\cdots\times S_{1}$ such that the denominator of $e_{n1}(\sigma(v+w'))$ is a factor of $$\prod_{r\in\Sigma\setminus\Sigma_{J}}((v+w')_{k_{r},i_{r}}-(v+w')_{k_{r},j_{r}})=\prod_{r\in\Sigma\setminus\Sigma_{J}}(v_{k_{r},i_{r}}-v_{k_{r},j_{r}})=P_{\Sigma\setminus\Sigma_{J}}(v).$$
Thus the coefficient of $\mathcal{D}_{R}T(v+w'+\sigma(\varepsilon_{n1}))$ in the expansion of $\mathcal{D}_{R}(P_{\Sigma\setminus\Sigma_{J}}(x)E_{n1}(T(x+w')))$ is $ev(v)(P_{\Sigma\setminus\Sigma_{J}}(x)e_{n1}(\sigma(x+w')))\neq 0$, so $v$ being regular ensures that the numerator of $e_{n1}(\sigma(x+w'))$ is nonzero after the evaluation. Hence we have 
$$\mathcal{D}_{R}T(v+w'')\rightarrow \mathcal{D}_{J}T(v+w')\rightarrow\mathcal{D}_{J}T(v+w),$$
where $w''=w'+\sigma(\varepsilon_{n1})$.\\

\

\noindent{\bf Proof of Theorem C.}\label{subsection: Proof of Th C}
 Since any submodule of a Gelfand-Tsetlin module is also a Gelfand-Tsetlin module (see Lemma 3.4 in \cite{FGR2}), for any $A\in V(T(v))$ the submodule $U\cdot A$ is a Gelfand-Tsetlin submodule of $V(T(v))$.

Let us consider the tableau $\mathcal{D}_{\emptyset}T(v+z)$  associated with $\sm_{v'}$ and denote by $M$ the submodule $U\cdot\mathcal{D}_{\emptyset}T(v+z)$. Set $$W=\{A\in V(T(v)) \; | \; A\in M,\ \text{ and }\ \mathcal{D}_{\emptyset}T(v+w)\notin U\cdot A\}.$$

If $W=\emptyset$ then $M$ is irreducible satisfying $M_{\sm_{v'}}\neq 0$. If $W\neq\emptyset$ then $N=\sum_{A\in W}U\cdot A$ is a nontrivial maximal proper submodule of $M$. Therefore, $M/N$ is an irreducible subquotient of $V(T(v))$ such that $(M/N)_{\sm_{v'}}\neq 0$. We can apply the same reasoning replacing $\mathcal{D}_{\emptyset}T(v+z)$ with $\mathcal{D}_{I}T(v+z) \neq 0$. Since the cardinality of the set $\{I\subseteq R \; | \; \mathcal{D}_{I}T(v+z)\neq 0\}$ is bounded by $2^{t}$, we obtain at most $2^{t}$ irreducible subquotients. This implies part (i).

To prove part (ii)  consider again the tableau  $\mathcal{D}_{\emptyset}T(v+z)$ associated with $\sm_{v'}$. In $V(T(v))$ we have $|\{J\subseteq R\; | \; \mathcal{D}_{J}T(v+z)\neq 0\}|=2^{t-k}$. If we construct an irreducible module $V$ as in part (i), as a quotient of   $U\cdot\mathcal{D}_{I}T(v+z)$, we have: $$\dim V_{\sm_{v'}}=|\{\mathcal{D}_{J}T(v+z)\; | \; \mathcal{D}_{J}T(v+z)\rightarrow \mathcal{D}_{I}T(v+z)\}|\leq 2^{t-k}.$$

Parts (iii) and (iv) follow directly from Theorem B, (i).

\end{document}